\documentclass[a4paper,11pt]{amsart}
\usepackage{mathrsfs}\usepackage{xcolor}
\usepackage[all]{xy}
\usepackage{amsmath,amssymb,amscd,bbm,amsthm,mathrsfs}
\usepackage{graphicx}
\newtheorem{thm}{Theorem}[section]
\newtheorem{lem}{Lemma}[section]
\newtheorem{cor}{Corollary}[section]
\newtheorem{prop}{Proposition}[section]
\newtheorem{rem}{Remark}[section]

\begin{document}
\numberwithin{equation}{section}

\title[On octonionic Hessian complex and Hartogs--Bochner extension...]{On octonionic Hessian complex and Hartogs--Bochner extension of octonionic pluriharmonic functions in two octonionic variables}
 
\author{Yun Shi$^a$ and Wei Wang$^b$}

\thanks{$^a$ Department of Mathematics, Zhejiang University of Science and Technology, Hangzhou 310023, China,
 Email: shiyun@zust.edu.cn}

\thanks{$^b$ 
 School of Mathematical Science, Zhejiang University  (Zijingang campus),  Hangzhou 310058, China, Email: wwang@zju.edu.cn}

 
 \thanks{The first author is  partially supported by Fundamental Research Funds of Zhejiang university of science and technology (No. 2025QN064),  Nature Science Foundation of Zhejiang province (No. LY22A010013); The second author is partially supported by National Nature Science Foundation in China  (No. 12371082).}
 \begin{abstract}
In this paper, we construct the octonionic Hessian complex, which is the octonionic version of the $\partial\bar \partial$-complex. By   proving  its ellipticity, we can  solve the non-homogeneous octonionic Hessian equations  under a compatibility condition. This is the octonionic version of the $\partial\bar\partial$-lemma. We also  obtain the boundary version of octonionic Hessian operator  and introduce the notion of a tangentially octonionic pluriharmonic function, which allows us to   establish the Hartogs--Bochner extension for tangentially octonionic  pluriharmonic  functions on the boundary of a domain with connected complement. This solves the octonionic version of Poincar\'e's problem for pluriharmonic extension on complex space.   The main difficulty  of this construction is the lack of associativity in octonionic case. However, we can overcome it by careful use of weak form of associativity including  Moufang identities and alternativity. 
\end{abstract}
    \maketitle
\section{Introduction}  
Analysis in one or several octonionic variables has been an active area of research  over the last decade. To name a few, H.-Y. Wang  and Ren \cite{WangR} have established the Bochner--Martinelli formula and  the Hartogs extension theorem for  the octonionic Dirac operator.  Colombo, Krau\ss har and Sabadini \cite{CKS} developed associated  octonionic Hardy and Bergman space theory via para-linear inner products, yielded explicit Szeg\H  o and Bergman kernel formulas for some typical  domains.  Octonionic Hardy spaces were also developed by Li \cite{Li}   using Clifford analysis.  
Xu  and Sabadini \cite{XS} extended  generalized partial-slice monogenic function theory to the   octonionic case. 
 Huo  and Ren \cite{HR} clarified the structure  of  octionionic Hilbert space.  Constales   and Krau\ss har \cite{CK} constructed the octonionic  Kerzman--Stein operator. Discrete octonionic analysis was also investigated by  Krau\ss har  and Legatiuk \cite{KL}.   Since    the   octonionic Dirac operator in one octonionic variable   is precisely the   Dirac operators of several  vector variables on $\mathbb R^8$ from $\mathbb S^+$ to $\mathbb S^-$ \cite{MS},   results for several Dirac operators (cf. \cite{Krump,Sabadini,SW3,SWW} etc.)  can be applied to several octonionic Dirac operators.  
  Recently, to study the octonionic version of Calabi--Yau theorem on octonionic  K\"ahler manifold, Alesker--Gordon \cite{AG} give the octonionic version of $\partial\bar\partial$-lemma. 
  
The  $\partial \bar \partial$-lemma is a fundamental result in complex analysis and complex geometry, i.e. one can solve \begin{align}
\partial \bar \partial u=f, 
\end{align} for a smooth $(1,1)$-form $f$ satisfying the compatibility  condition \begin{align}
\partial f=0=\bar \partial f.
\end{align}  It implies the local existence of K\"ahler potential. These two operators are first two operators in the $\partial \bar \partial$-complex
\begin{equation}\begin{aligned}\label{ddbarc} \Gamma\left(\Omega,\Lambda^{0,0}\right) \xrightarrow{\partial\bar \partial}\Gamma\left(\Omega,\Lambda^{1,1}\right)\xrightarrow{{\rm d}}\Gamma\left(\Omega,\Lambda^{1,2}\right) \oplus \Gamma\left(\Omega,\Lambda^{2,1}\right)\xrightarrow{ }\cdots,\end{aligned}\end{equation}
where $\Lambda^{r,s}$ is the space of smooth $(r,s)$-forms. Its cohomology, the Bott--Chern cohomology, plays an important role in the study of complex manifolds \cite{AT}.   To state  the result of  Alesker and Gordon, denote the space of  $2\times2$ \emph{octonionic Hermitian matrices} by \begin{align*}
\mathcal H_2(\mathbb O):=\left\{\left(\begin{matrix}a&q\\\bar q&b\end{matrix}\right):a,b\in\mathbb R,q\in\mathbb O\right\}.
\end{align*}
Denote a point of  $\mathbb O^2$ by   $\mathbf x=\left( x_1, x_2\right),$ where $  x_j=\sum_{a=0}^7x_j^{a}e_a\in\mathbb O,j=1,2,$ and $\left\{e_a\right\}$ is a basis of $\mathbb O $ satisfying the relation \begin{align}\label{mult}e_{0}^2=1,\ e_a^2=-1,\ e_ae_b=-e_be_a,\end{align} for $a\neq b, a,b=1,\dots,7.$ The conjugate of $ x_j$ is $\bar{ x}_j=x^0_j-\sum_{a=1}^7x_j^ae_a.$ %and ${\rm Re}\  x_j=x_j^0,$ ${\rm Im}\ x_j=\sum_{a=1}^7x_j^ae_a.$ 
   For a $\mathbb O$-valued      smooth function $F$, we can define   \emph{octonionic left  and right Dirac operators}
\begin{equation}\begin{aligned}
 {\partial_{\bar j}}F:=\sum_{a=0}^7   e_a\frac{\partial F}{\partial x_j^a}, 
\qquad \qquad
 F\overleftarrow{\partial_{\bar j}}:=\sum_{a=0}^7\frac{\partial F}{\partial x_j^a}  e_a,
\end{aligned}\end{equation}
for $j=1,2,$ respectively. They are generalizations of the Cauchy--Riemann operator in several complex variables and Cauchy--Fueter operator in several quaternionic variables. We also need conjugate version:   
\begin{equation}\begin{aligned}
 {\partial_{  j}}F=\sum_{a=0}^7   \bar e_a\frac{\partial F}{\partial x_j^a},\qquad \qquad F\overleftarrow{\partial_j}=\sum_{a=0}^7 \frac{\partial F}{\partial x_j^a}\bar e_a.
\end{aligned}\end{equation} In the octonionic case, the $\partial\bar \partial$ operator is replaced by \emph{octonionic Hessian operator} \begin{align} \label{D0}
\mathcal D_0f:=Hess_{\mathbb O}(f)=\left(\begin{matrix}\Delta_1f &{\partial_{ \bar 1}} {\partial_{ 2}}f\\{\partial_{ \bar 2}} {\partial_{ 1}}f&\Delta_2f \end{matrix}\right), \end{align}   for $f\in\Gamma(\Omega,\mathbb R)$.
To solve the non-homogeneous equation \begin{align}\label{duf}
\mathcal D_0f=u, 
\end{align}  
  with  $u=(u_{\bar ij})\in\Gamma\left(\Omega,\mathcal H_2(\mathbb O)\right), $ Alesker--Gordon \cite{AG} found the compatibility condition \begin{align}\label{D1} \mathcal D_1 u:=&\left(\begin{matrix}u_{\bar 1 2}\overleftarrow{\partial_{ \bar 2}}-{\partial_{\bar 1}}u_{\bar 2 2}\\u_{ \bar 2 1}\overleftarrow{\partial_{\bar 1}}-{\partial_{\bar 2}}u_{\bar 11}\end{matrix}\right)=0,
\end{align}
by using   computer program  Macaulay2. They also used this program     to  demonstrate the associated sequence of polynomial modules to be exact, which implies  (\ref{duf}) is solvable on an open convex domain under the condition (\ref{D1}) (cf. e.g. \cite[Theorem 7.6.13]{Homander1}).   

 In this paper, we will   solve  the non-homogeneous equation (\ref{duf}) under compatibility condition (\ref{D1})  by solving the associated Hodge-Laplacian equations. This method requires knowledge of  the next differential operator $\mathcal D_2.$ In fact, we   construct the whole    differential complex on a domain $\Omega\subset\mathbb O^2:$  
 \begin{equation}\begin{aligned}\label{co}
0\rightarrow \Gamma\left(\Omega,\mathcal V_0\right)\xrightarrow{\mathcal{D}_{0}}\Gamma\left(\Omega,\mathcal V_1\right) \xrightarrow{\mathcal{D}_{1}} \Gamma\left(\Omega, \mathcal V_2\right)\xrightarrow{\mathcal{D}_{2}} \Gamma\left(\Omega, \mathcal V_3 \right)\xrightarrow{\mathcal{D}_{3}} \Gamma\left(\Omega,\mathcal V_4 \right)\xrightarrow{\mathcal{D}_{4}} \Gamma\left(\Omega, \mathcal V_5  \right)\rightarrow 0,
\end{aligned}\end{equation}where \begin{align}\mathcal V_0=\mathcal V_5=\mathbb R,\quad \mathcal V_1=\mathcal V_4=\mathcal H_2(\mathbb O),\quad \mathcal V_2=\mathcal V_3=\mathbb O^2,\end{align} i.e. $\mathcal D_{j+1}\circ\mathcal D_j=0$ for each $j$. It is the octionionic version of $\partial\bar \partial$ complex in several complex variables and $0$-Cauchy--Fuerer complex in several quaternionic variables \cite{WanW}.
Moreover, we   prove 
\begin{thm}\label{exact}The  complex {\rm(\ref{co})} is   elliptic. 
\end{thm}\noindent It means its  symbol sequence    
\begin{equation}\begin{aligned}\label{simco}
0\rightarrow\mathbb R\xrightarrow{\sigma_{0}(\xi)}  \mathcal H_2(\mathbb O)\xrightarrow{\sigma_{1}(\xi)} \mathbb O^2\xrightarrow{\sigma_{2}(\xi)} \mathbb O^2  \xrightarrow{\sigma_{3}(\xi)}  \mathcal H_2(\mathbb O)\xrightarrow{\sigma_{4}(\xi)}  \mathbb R\rightarrow0,
\end{aligned}\end{equation}  exact for each $\xi\in\mathbb O^2\setminus\{\mathbf 0\}$, where $\sigma_j(\xi)$ is the symbol of the operator $\mathcal D_j$.

  The operator  $\mathcal D_2$  in (\ref{co}) is 
{\small\begin{equation}\begin{aligned}\label{D} 
  \mathcal D_2v:=&\left(\begin{matrix} \left( {\partial_{\bar 1}} {\partial_2} \right)v_2- \left({\partial_{\bar 1}}\bar v_2\right)\overleftarrow{\partial_{\bar 2}}+2{\rm Im}\left(v_1\overleftarrow{\partial_1}\right) \overleftarrow{\partial_{\bar 1 }} \\ \left( {\partial_{\bar 2}} {\partial_1} \right)v_1- \left({\partial_{\bar 2}}\bar v_1\right)\overleftarrow{\partial_{\bar 1}}+2{\rm Im}\left(v_2\overleftarrow{\partial_2}\right) \overleftarrow{\partial_{\bar 2}} \end{matrix}\right),\quad {\rm for}\ v=\left(\begin{matrix}v_1\\v_2\end{matrix}\right) \in\Gamma\left(\Omega, \mathbb O^2\right),
\end{aligned}\end{equation}}while  $\mathcal D_3$ and $\mathcal D_4$ as formal adjoint of $\mathcal D_1$ and $\mathcal D_0$ respectively, are given by  \begin{equation}\begin{aligned}\label{D34}\mathcal D_3t: =& \left(\begin{matrix} {\rm Re}\left(\partial_{ 2}t_2\right) &-\frac{{\partial_{\bar1}} \bar t_2+ t_1 \overleftarrow{\partial_2}}{2}\\-\frac{t_2 \overleftarrow{\partial_1}  +  {\partial_{\bar2}} \bar t_1 }{2} &{\rm Re}\left({\partial_{ 1}}t_1\right)\end{matrix}\right),   \quad {\rm for}\ t= \left(\begin{matrix}t_1\\t_2\end{matrix}\right) \in\Gamma\left(\Omega, \mathbb O^2\right),\\\mathcal D_4s: =&  
 \Delta_1 s_{\bar11}+\Delta_2s_{\bar 22}+2{\rm Re}\left({\partial_{2}}s_{\bar 21} \overleftarrow{\partial_{\bar 1}} \right), \quad {\rm for}\ s=\left(\begin{matrix}s_{\bar 1 1} &s_{\bar 1 2} \\s_{ \bar 2 1} &s_{\bar 2 2} \end{matrix}\right) \in\Gamma\left(\Omega,\mathcal H_2(\mathbb O)\right), 
\end{aligned}\end{equation}
where $\Delta_j = \sum_{a=0}^7 \partial^2_{x_{j}^a},j=1,2$ is Laplacian. To construct operators $\mathcal D_j$'s, the main difficulty is the  lack of associativity in octonionic case. However, we can overcome it by carefully using the following weak form of associativity: alternativity (theorem of Artin) and the Moufang identities (cf. e.g. \cite[Proposition 1.4.1]{SpringerV} \cite{Schafer}).  
\\
{\bf Theorem A} {\it 
 The subalgebra generated by any two elements $u, v$ of an alternative algebra  is associative, i.e.  
\begin{align*}
 (uv)u = u(vu),
\end{align*}which can be written as $uvu$ without confusion.} 
\\
{\bf Theorem B} {\it 
   The composition algebra $\mathbf C$ {\rm(}octonions over any number field $F$  form a composition algebra{\rm)} obey some weak associative laws,  
\begin{align}
\label{M1}(uvu)x = u(v(ux)),\\ \label{M2}x(uvu) = ((xu)v)u,\\ \label{M3}u(xy)u = (ux)(yu),
\end{align}
for any $u,v,x,y\in \mathbf C.$ 
}

 A function $f$ on a domain in $\mathbb C^n$ is said to be  \emph{pluriharmonic} if it  is harmonic on every complex line. It is equivalent to requiring  $$\partial \bar \partial f=0.$$   $\partial\bar\partial$-complex also have the following important application to a problem set by Poincar\'e \cite{Poincare}:    to find necessary and sufficient conditions for a  real-valued  function $u$ given on the boundary of a bounded domain $\Omega\in\mathbb C^n$ is the trace on $\partial \Omega$ of a pluriharmonic function in $\Omega.$ The history and some  known results about this problem can be found in \cite{Fichera}.  It was solved by Bedford on the unit ball in $\mathbb C^n$ \cite{Bedford} and then by Bedford--Federbush \cite{BF} for general domain with $C^3$ boundary.  
  
  A function $f$  is said to be  \emph{quaternionic pluriharmonic} if it  is harmonic on every quaternionic  line.  A function $f$  is said to be  \emph{octonionic  pluriharmonic} (OPH) if it  is harmonic on every octonionic  line. Equivalently,    $f$ is   an  {OPH   function} on $\Omega\in\mathbb O^2$ if \begin{align*}
\mathcal D_0f(\mathbf x)=0,
\end{align*}for any $\mathbf x\in\Omega.$ The space of all OPH functions on $\Omega$ is denoted by OPH$ (\Omega).$ 
We will apply octonionic Hessian complex to show Hartogs--Bochner extension for OPH function in two variables.

 As quaternionic Hessian can be used to define quaternionic   Monge--Amper\`e operator \cite{Alesker1,Alesker2}, octonionic Hessian can be used to define the \emph{octonionic Monge--Amp\`ere operator}  
\begin{align}
{\rm det}\left(Hess_{\mathbb O}(f)\right).
\end{align} It plays an important role in the study of plurisubharmonic functions on the octonionic plane $\mathbb O^2$ introduced by Alesker \cite{Alesker3} and applied to the theory of valuations on convex sets.
   Alesker--Gordon \cite{AG} also  
  introduced a class of Riemannian metrics on ${\rm GL}_2(\mathbb O)$-affine manifolds that are octonionic analogues of K\"ahler metrics on complex manifolds and of HKT-metrics of hypercomplex manifolds and established  an octonionic version of the Calabi--Yau theorem on octonionic  K\"ahler manifolds by studying octonionic Monge--Amp\`ere equation.

As in the quaternionic case \cite{Wa10}-\cite{wang29} and the case of several vector variables   \cite{SW3,SWW}, we can apply the  PDE method of several complex variable to study this complex. We show the non-homogeneous octonionic Hessian  equation (\ref{duf})  
has  compactly supported solution if $u$ is compactly supported and satisfies the compatibility condition  (\ref{D1}). 
 \begin{thm}\label{t31}
Suppose that $u\in L^2\left(\mathbb O^{2},\mathcal V_{j+1}\right),j=0,1,\dots,4,$ satisfies the compatibility condition $\mathcal D_{j+1}u=0$  
 in the sense of distributions. Then there exists a function $f\in W^{1,2}\left(\mathbb O^{2},\mathcal V_j\right)$ satisfying the non-homogeneous   equation  
 $\mathcal D_jf=u$. Furthermore, if $u\in C_c\left(\mathbb O^{2},\mathcal V_{j+1}\right)$ with $\mathcal D_{j+1}u=0$ in the sense of distributions,  then there exists a function $f\in C_c\left(\mathbb O^{2},\mathcal V_j\right)\cap W^{1,2}\left(\mathbb O^{2},\mathcal V_j\right)$ satisfying $\mathcal D_jf=u$  
  and vansihing on the unbounded connected component of $\mathbb O^{2}\setminus {\rm supp}\ u.$
\end{thm}\noindent This theorem  implies  the Hartogs extension theorem for OPH functions (Corollary \ref{hart}).

Andreotti--Nacinovich \cite{AN} constructed a boundary version $\left(\partial\bar \partial\right)_b$ of the operator $\partial\bar\partial,$ which is also an operator of second order. For a couple of functions $u_0$ and $u_1$ satisfying $\left(\partial\bar \partial\right)_b\left(u_0+\rho u_1\right)=0$ on the boundary, where $\rho$ is  the defining function of a domain $\Omega$ with connected complement, then there exists a pluriharmonic function $U$ such that $U=u_0+\rho u_1+O\left(\rho^2\right)$ on the boundary. Equivalently, a $C^2$ function $u$ is quaternionic  pluriharmonic if and only if $\Delta u = 0,$ where $\Delta$ is the $0$-Cauchy--Fueter operator. The second named  author established  the Hartogs--Bochner extension for $k$-regular functions including $k=0$ \cite{wang29}. The case $k = 1$   was   proved by Maggesi--Pertici--Tomassini \cite{MPT} first.  

To investigate OPH functions on a domain, we need  the boundary  complex of the octonionic Hessian complex (\ref{co}). Without loss of generality, we can assume the real  defining function of the domain can be written locally as
 \begin{align}\label{defin'}
\rho(\mathbf x)=x^0_1-\phi\left(x^1_1,\dots x^7_1,x^0_2,\dots,x^7_2\right),
 \end{align} for $\mathbf x\in \mathbb O^2.$  
 By using the construction  of the boundary complex of a general differential complex  (see e.g. \cite{Andreotti,AN,Naci,Nacinovich85}),  we get the \emph{tangential octonionic Hessian  operators}:
  \begin{equation}\begin{aligned}\label{mathscrD}
 {\mathscr D}_0:\Gamma\left(\partial \Omega,\mathbb R^2\right)&\rightarrow \Gamma\left(\partial \Omega,  \mathcal H'_2(\mathbb O) 
 \right),\\\hat f=\left(u_0,u_1\right)&\mapsto \mathscr D_0\hat f, 
  \end{aligned}\end{equation} where \begin{align}\label{h2o'}\mathcal H'_2(\mathbb O):=\left\{\left(\begin{matrix}0&q\\\bar q&b\end{matrix}\right): b\in\mathbb R,q\in\mathbb O\right\}.\end{align}
It is  the counterpart  $\left(\partial\bar\partial\right)_b$ on the boundary of a domain in $\mathbb C^n.$   $\hat f\in \Gamma\left(\partial \Omega,\mathbb R\right)$ is said to be \emph{tangentially octonionic  pluriharmonic {\rm(}tangentially OPH{\rm)}} if \begin{align*} {\mathscr D}_0\hat f=0.
\end{align*}
  We   establish  the Hartogs--Bochner extension  for tangentially OPH functions.
\begin{thm}\label{HBE} Suppose that   $\Omega$ is a bounded domain in $\mathbb O^{2}$ with smooth boundary such that $\mathbb O^{2}\setminus \overline \Omega$ is connected. If $f=\left(u_0,u_1\right)$ is a smooth tangentially OPH  function on $\partial \Omega$, then there exists a OPH function $\tilde f$ in $ \Omega $ smooth up to the boundary such that $\tilde f =u_0+\rho u_1 $   (${\rm mod}\ \rho^2  
$).
\end{thm}
The paper is organized as follows. In Section 2, we give the preliminaries on octonions and deduce the adjoint operator of octonionic Dirac operators $\mathcal D_0,\mathcal D_1$.  In Section 3, we construct the complex    {\rm(\ref{co})} with operators given by   (\ref{D0}), (\ref{D1})  and  (\ref{D})-(\ref{D34}), and show it to be     elliptic. In Section 4, we study   the boundary complex of the octonionic Hessian  complex and obtain the boundary version of octonionic Hessian  operators. In Section 5, we solve the non-homogeneous octonionic Hessian equation (\ref{duf}) under the compatibility condition (\ref{D1}), by solving the associated Hodge-Laplacian equations of fourth order. Finally,  we prove the Hartogs--Bochner extension for tangentially OPH  functions. In the Appendix, we  prove the complexified version of Lemma \ref{asso}.

\section{Preliminaries}  
\subsection{Octonions} Let $\mathbb O_{\mathbb F} $ denote the octonions over a number field $\mathbb F$. 
Any element   $x\in\mathbb O_{\mathbb F}$   can written as   $x=\sum_{a=0}^7z_ae_a$ with $z_a\in {\mathbb F}.$  Its conjugation is defined by  \begin{align}\label{conj}\bar x=\sum_{a=0}^7z_a\bar e_a.\end{align} Denote   ${\rm Re}\ x=z_0,$ ${\rm Im}\ x=\sum_{a=1}^7z_a e_a $ and  $|x|^2=\sum_{a=0}^7z_a^2.$ Note that the conjugate is $\mathbb F$-linear and  ${\rm Re}\ x$ may not always be real number for octonions over a field $\mathbb F.$  
We write $\mathbb O=\mathbb O_{\mathbb R}$ for  simplicity. We refer to \cite{SpringerV} for a comprehensive reference on octonionic algebras over any field.

 An octonion $  x\in\mathbb O$ can be written as $  x=  a+  be_4,$ where $( a,  b)\in\mathbb H^2.$ Then the multiplication can be written as (cf. e.g. \cite{AG,Baez})
$$(  a,   b)(  c,   d) = \left(  a  c -\bar{ { d}}  b,   d  a +   b \bar{ {c}}\right).$$ 
The octonion  algebra $\mathbb{O}$ is neither commutative nor associative. For $  x,y,z\in\mathbb O_{\mathbb F},$ define an {\it associator}: $ \{x,y,z\}:=(xy)z-x(yz).$ %By Moufang identities in Proposition  \ref{Mou},  
The octonion  algebra is {\it alternative}, i.e. $  \{u,v,u\}=\{u,u,v\}=\{v,u,u\}=0$ for any $  u,v\in\mathbb O_{\mathbb F}$ by Theorem A.
 Define a    scalar product $\langle x,y\rangle $ on $\mathbb O_{\mathbb F}$ by 
\begin{align*}
   \langle  x,y \rangle:={\rm Re}\left(x\bar y\right).
\end{align*}
For ${  u=\left(\begin{matrix}u_1\\u_2\end{matrix}\right), v=\left(\begin{matrix}v_1\\v_2\end{matrix}\right)}\in\mathbb O_{\mathbb F}^2,$ define \begin{align*}
 \langle    u,  v \rangle: 
 =\sum_{i=1}^2{\rm Re}\left(  u_i\bar{   v}_i\right),
\end{align*}
and for $H=\left(H_{\bar ij}\right) 
\in\mathcal H_2(\mathbb O_{\mathbb F}),$ define 
\begin{equation}\begin{aligned}\label{inner}
\left\langle H,H'\right\rangle :=&\left\langle H_{\bar11},H'_{\bar11}\right\rangle +\left\langle H_{\bar22},H'_{\bar22}\right\rangle +\left\langle H_{\bar12},H'_{\bar12}\right\rangle +\left\langle H_{\bar21},H'_{\bar21}\right\rangle \\ =& H_{\bar11}H'_{\bar11}+ H_{\bar22}H'_{\bar22}+2{\rm Re}\left(H_{\bar12}H'_{\bar21}\right).
\end{aligned}\end{equation}

\begin{lem}\label{asso}
Let $  x, y, z\in$ $\mathbb O$ or $\mathbb O_{\mathbb C}$. Then\\
{\rm (i)} $  {\rm Re} (x) = {\rm Re} \left(\bar x\right);$\\{\rm (ii)} $  {\rm Re} (xy) = {\rm Re} (yx);$\\
{\rm (iii)}
 $  {\rm Re} ((xy)z) = {\rm Re} (x(yz))$ {\rm(}this  number will be denoted by $  {\rm Re} (xyz){\rm)};$\\
{\rm(iv)}  
Any subalgebra of $\mathbb O$ or $\mathbb O_{\mathbb C}$ generated by any two elements and their conjugates is associative. 
\end{lem}
See  \cite[Lemma 2.1, 2.2]{AG}  for this lemma for real octonions, we will prove the  complexified version  in Appendix \ref{app}.
\begin{prop}
\begin{align}\partial_j\partial_{\bar j}=\partial_{\bar j}\partial_j=\overleftarrow{\partial_j}\overleftarrow{\partial_{\bar j}}=\overleftarrow{\partial_{\bar j}}\overleftarrow{\partial_j}=\Delta_j.\end{align}
\end{prop}
\begin{proof}
Note that 
\begin{equation}\begin{aligned}
\partial_j\partial_{\bar j} =\sum_{a,b=0}^7 \bar e_be_a\frac{\partial^2 }{\partial x_j^b\partial x_j^a}=\frac{\partial^2 }{\partial x_j^0\partial x_j^0}-\sum_{a,b=1}^7   e_be_a\frac{\partial^2 }{\partial x_j^b\partial x_j^a} 
=\Delta_j ,
\end{aligned}\end{equation}by  (\ref{mult}).  
Similarly, we can prove the other cases.
\end{proof}

\subsection{The formal adjoint operators}
 We have  the following formulae of integration by parts for ${\partial_{\bar j}},{\partial_{  j}},\overleftarrow{\partial_{\bar j}},\overleftarrow{\partial_{ j}}.$
\begin{prop}\label{ibp}Suppose that $\Omega$ is a bounded domain in $\mathbb O^{2}.$ For any $f\in C^1\left(\overline \Omega,\mathbb O\right)$ and $h\in C^1\left(\overline \Omega,\mathbb O\right),$ we have
\begin{equation}\begin{aligned}\label{bp}
\int_{\Omega}\left\langle{\partial_{\bar j}}f,h\right\rangle {\rm d}V=&-\int_\Omega\left\langle f, {\partial_{  j}}h\right\rangle {\rm d}V+\int_{\partial \Omega}\left\langle \rho_{\bar j}\cdot f, h \right\rangle \frac{{\rm d}S}{|{\rm grad}\rho|},\\ \int_{\Omega}\left\langle{\partial_{ j}}f,h\right\rangle {\rm d}V=&-\int_\Omega\left\langle f, {\partial_{\bar  j}}h\right\rangle {\rm d}V+\int_{\partial \Omega}\left\langle \rho_{  j}\cdot f, h \right\rangle \frac{{\rm d}S}{|{\rm grad}\rho|},\\ \int_{\Omega}\left\langle f\overleftarrow{\partial_{ j}},h\right\rangle {\rm d}V=&-\int_\Omega\left\langle f, h\overleftarrow{\partial_{\bar j}}\right\rangle {\rm d}V+\int_{\partial \Omega}\left\langle f\cdot\rho_{  j}, h \right\rangle \frac{{\rm d}S}{|{\rm grad}\rho|},\\ \int_{\Omega}\left\langle f\overleftarrow{\partial_{\bar j}},h\right\rangle {\rm d}V=&-\int_\Omega\left\langle f, h\overleftarrow{\partial_{ j}}\right\rangle {\rm d}V+\int_{\partial \Omega}\left\langle f\cdot  \rho_{\bar j}, h \right\rangle \frac{{\rm d}S}{|{\rm grad}\rho|},
\end{aligned}\end{equation}where $\rho$ is a defining function of $\Omega$ and $$  \rho_j:=\partial_j\rho=\rho\overleftarrow{\partial_j},\quad \rho_{\bar j}:=\partial_{\bar j}\rho=\rho\overleftarrow{\partial_{\bar j}},\quad j=1,2.$$ 
\end{prop}
\begin{proof} 
 It is easy to see that
\begin{equation*}\begin{aligned}
\int_\Omega\left\langle{\partial_{\bar j}}f,h\right\rangle {\rm d}V=&\sum_a  \int_{\Omega}{\rm Re}\left(e_a\frac{\partial f}{\partial x_j^a}\cdot \bar h\right){\rm d}V=\sum_a  \int_{\Omega}{\rm Re}\left(\frac{\partial f}{\partial x_j^a}\cdot \bar h e_a\right){\rm d}V\\=&-\sum_a  \int_{\Omega}{\rm Re}\left(f\cdot \frac{\partial\bar h}{\partial x_j^a} e_a\right){\rm d}V+\sum_a  \int_{\partial\Omega}{\rm Re}\left(\frac{\partial\rho}{\partial x_j^a}f\cdot \bar he_a\right)\frac{{\rm d}S}{|{\rm grad}\rho|}\\=&-\sum_a  \int_{\Omega}{\rm Re}\left(f\cdot \overline{\bar e_a\frac{\partial  h}{\partial x_j^a}}\right){\rm d}V+\sum_a  \int_{\partial\Omega}{\rm Re}\left(e_a\frac{\partial\rho}{\partial x_j^a}f\cdot \bar h\right)\frac{{\rm d}S}{|{\rm grad}\rho|}\\ =&- \int_\Omega\left\langle f, {\partial_{  j}}h\right\rangle {\rm d}V+\int_{\partial \Omega}\left\langle \rho_{\bar j}\cdot f, h \right\rangle \frac{{\rm d}S}{|{\rm grad}\rho|}, 
\end{aligned}\end{equation*}by  using Lemma \ref{asso} (ii)-(iii) repeatedly. Similarly, we can prove the other cases.
\end{proof}
Define  the formal adjoints of $\partial_{\bar j},\partial_{ j},\overleftarrow{\partial_j},\overleftarrow{\partial_{\bar j} }$ by \begin{equation*}\begin{aligned}
\int_{\mathbb O^2}\left\langle f,{\partial_{\bar j}^*}g\right\rangle {\rm d}V:=&\int_{\mathbb O^2}\left\langle{\partial_{\bar j}}f,g\right\rangle {\rm d}V,\\ \int_{\mathbb O^2}\left\langle f,{\partial_{  j}^*}g\right\rangle {\rm d}V:=&\int_{\mathbb O^2}\left\langle{\partial_{  j}}f,g\right\rangle {\rm d}V,\\\int_{\mathbb O^2}\left\langle f,g{\left(\overleftarrow{\partial_{\bar j}}\right)^*}\right\rangle {\rm d}V:=&\int_{\mathbb O^2}\left\langle f{ \overleftarrow{\partial_{\bar j}} },g\right\rangle {\rm d}V,\\ \int_{\mathbb O^2}\left\langle f,g{\left(\overleftarrow{\partial_{ j}}\right)^*}\right\rangle {\rm d}V:=&\int_{\mathbb O^2}\left\langle f{ \overleftarrow{\partial_{j}} },g\right\rangle {\rm d}V,\end{aligned}\end{equation*}for $f,g\in C_c^1\left(\mathbb O^2,\mathbb O\right).$ Then we have the following corollary directly by Proposition \ref{ibp}.

 \begin{cor}The formal adjoint of $\partial_{\bar j},\partial_{ j},\overleftarrow{\partial_{j} },\overleftarrow{\partial_{\bar j} }$ are 
 \begin{align}\label{adj}
\partial_{\bar j}^*=-{\partial_{ j}},\quad\partial_{j}^*=-{\partial_{\bar j}}, \quad \left(\overleftarrow{\partial_{j}}\right)^*= -\overleftarrow{\partial_{\bar j}},\quad \left(\overleftarrow{\partial_{\bar j}^*}\right)^*=-\overleftarrow{\partial_{j}}.
\end{align}
 \end{cor}
 
Note that ${\partial_{ \bar 1}} {\partial_{ 2}}f={\partial_{ \bar 1}}f\overleftarrow{\partial_{ 2}}$ for $ \mathbb R$-valued function $f.$ In the sequel,   we will write $$Hess_{\mathbb O}(f)=\left(\begin{matrix}\Delta_1f &{\partial_{ \bar 1}} f\overleftarrow{{\partial_{ 2}}}\\{\partial_{ \bar 2}}f\overleftarrow{{\partial_{ 1}}}&\Delta_2f \end{matrix}\right)$$ for convenience.

\begin{prop}\label{Dstar}The  formal adjoint of $ \mathcal D_0$ and $\mathcal D_1 $ are given by
\begin{equation}
\begin{aligned}\label{ddd}
\mathcal D_0^*u =&\Delta_1u_{\bar11}+\Delta_2u_{\bar22} +2{\rm Re}\left({\partial_{2}}u_{\bar21} \overleftarrow{\partial_{ \bar 1}}\right),\\ \mathcal D_1^*v=&\left(\begin{matrix} {\rm Re}\left(   {\partial_{ 2}}v_2 \right) &-\frac{{{\partial_{\bar 1}}\bar v_2 }+{v_1\overleftarrow{\partial_{ 2}}}}{2} \\-\frac{v_2\overleftarrow{\partial_{ 1}}+ \partial_{\bar 2}\bar v_1 }{2} &{\rm Re}\left( {\partial_{ 1}}v_1\right)  \end{matrix}\right), \end{aligned}\end{equation} for  $u=\left(u_{\bar ij}  \right)\in C_c^\infty\left(\Omega,\mathcal H_2(\mathbb O)\right),$   $v=\left(\begin{matrix}v_1\\  v_2\end{matrix}\right)\in C_c^\infty\left(\Omega,\mathbb O^2\right),$
respectively.
\end{prop}
\begin{proof}
For $f\in C_c^\infty\left(\Omega,\mathbb R\right),$ we have \begin{equation*}\begin{aligned}
\left(\mathcal D_0f,u\right)= 
&\int_{\mathbb O^2}\left\langle\Delta_1f, u_{\bar 1 1} \right\rangle+\left\langle{\partial_{ \bar 1}}f\overleftarrow{\partial_{ 2}}, u_{\bar 1 2} \right\rangle+\left\langle{\partial_{ \bar 2}}f\overleftarrow{\partial_{ 1}}, u_{\bar 21} \right\rangle+\left\langle\Delta_2f, u_{\bar 2 2} \right\rangle{\rm d}V\\=&\int_{\mathbb O^2}\left\langle f,\Delta_1u_{\bar11}  +\left({\partial_{ 1}}u_{\bar12}\right) \overleftarrow{\partial_{ \bar 2}}+{\partial_{ 2}}\left(u_{\bar21} \overleftarrow{\partial_{ \bar 1}}\right)+\Delta_2u_{\bar22} \right\rangle{\rm d}V\end{aligned}\end{equation*}by using (\ref{adj}). Thus (\ref{ddd}) follows. Similarly,
\begin{equation*}\begin{aligned} \left(\mathcal D_1u,v\right)=&\int_{\mathbb O^2}\left\langle u_{ \bar 1 2}\overleftarrow{\partial_{\bar 2}}-{\partial_{ \bar 1}}u_{\bar 22},v_1\right\rangle{\rm d}V+\int_{\mathbb O^2}\left\langle u_{ \bar 21}\overleftarrow{\partial_{ \bar 1}}-{\partial_{ \bar 2}}u_{\bar 1 1},v_2\right\rangle{\rm d}V\\=&\int_{\mathbb O^2}\left(\left\langle u_{\bar 1 2},-v_1\overleftarrow{\partial_{ 2}}\right\rangle + \left\langle u_{ \bar 2 2},{\partial_{ 1}}v_1\right\rangle  + \left\langle u_{ \bar 2 1},-v_2\overleftarrow{\partial_{ 1}}\right\rangle + \left\langle u_{ \bar 11},{\partial_{ 2}}v_2\right\rangle\right){\rm d}V,
\end{aligned}\end{equation*}
by using (\ref{adj}) again,
where 
\begin{equation*}\begin{aligned}\label{2.7}
\left\langle u_{ \bar 12},-v_1\overleftarrow{\partial_{ 2}}\right\rangle+\left\langle u_{ \bar 2 1},-v_2\overleftarrow{\partial_{ 1}}\right\rangle =&-{\rm Re}\left( u_{ \bar 1 2}\cdot\overline{v_1\overleftarrow{\partial_{ 2}}}\right)-{\rm Re}\left(\overline{u_{ \bar 1 2}}\cdot\overline{v_2\overleftarrow{\partial_{ 1}}}\right)\\=&-{\rm Re}\left( u_{ \bar 1 2}\cdot\overline{v_1\overleftarrow{\partial_{ 2}}}\right)-{\rm Re}\left( {v_2\overleftarrow{\partial_{  1}}}\cdot  {u_{ \bar 1 2}} \right) 
\\=&-{\rm Re}\left( u_{ \bar 1 2}\cdot\left(\overline{v_1\overleftarrow{\partial_{ 2}}} +v_2\overleftarrow{\partial_{ 1}}\right)\right)=\left\langle u_{\bar 1 2},-\left({v_1\overleftarrow{\partial_{  2}}} + \partial_{\bar 1}\bar v_2 \right)\right\rangle, 
\end{aligned}\end{equation*}
by using Lemma \ref{asso} (i)-(iii) and  $
\left\langle u_{ \bar 1 1},{\partial_{ 2}}v_2\right\rangle =  \left\langle u_{ \bar 1 1}, { \rm Re}\left({\partial_{  2}}v_2\right)   \right\rangle$ since $u_{\bar 11}$ is real.
Thus (\ref{ddd}) follows.  
The proposition is proved. 
\end{proof}
   
\begin{cor}\label{p23}
For any $f\in C^\infty\left(\mathbb O^2,\mathbb R\right),$ we have
\begin{align*}
\mathcal D_0^*\mathcal D_0f=\Delta^2 f,
\end{align*}
where $\Delta=\Delta_1+\Delta_2 $ is the standard Laplacian operator on $\mathbb O^{2}.$
\end{cor}
\begin{proof}
By (\ref{ddd}), we have
\begin{equation}\begin{aligned}\label{dsd}
\mathcal D_0^*\mathcal  D_0 f=&\mathcal D_0^*\left(\begin{matrix}\Delta_1f &{\partial_{ \bar 1}}f\overleftarrow{\partial_{ 2}}\\{\partial_{ \bar 2}}f\overleftarrow{\partial_{ 1}}&\Delta_2f \end{matrix}\right)\\=& \Delta_1^2f + \Delta_2^2f +2{\rm Re}\left(\partial_2\left({\partial_{\bar 2 }}f\overleftarrow{\partial_{  1}}\right)  \overleftarrow{\partial_{\bar1}}\right) \\=&\left(\Delta_1^2+\Delta_2^2+\Delta_2\Delta_1+\Delta_1 \Delta_2\right)f=\Delta^2f,
\end{aligned}\end{equation}
since
 $${\rm Re}\left[\partial_2\left({\partial_{\bar 2 }}f\overleftarrow{\partial_{  1}}\right)  \partial_{\bar1}\right]={\rm Re}\left[\left(\left(\partial_2 {\partial_{\bar 2 }}f\right)\overleftarrow{\partial_{  1}}\right)  \partial_{\bar1}\right]={\rm Re}\left[\left( \Delta_2f \overleftarrow{\partial_{  1}}\right)  \partial_{\bar1}\right]=\Delta_2\Delta_1f,$$ by applying  Lemma \ref{asso}.
The corollary  is proved.
\end{proof}
\begin{cor}\label{p34}
An OPH  function    $f $ on a domain $\Omega\subset\mathbb O^{2}$ is real analytic.
\end{cor}
\begin{proof}
By (\ref{dsd}), we have $\mathcal D_0^*\mathcal D_0f=\Delta^2 f=0,$
in the sense of distributions. Thus $f$ is biharmonic and so it is  real analytic at each point of  $\Omega.$
\end{proof}

 \section{Octonionic Hessian complex and its ellipticity}
\subsection{Octonionic Hessian complex} 
For a $\mathbb O$-valued function  $f\in L^1\left(\mathbb O^{2},\mathbb O\right),$ its Fourier transform  is defined by 
$$\hat f(\xi)=\int_{\mathbb O^2}f(x)e^{-\mathbf i\langle x, \xi\rangle}{\rm d}x,$$ for $\xi=\left(\xi_1,\xi_2\right)\in \mathbb O^2.$ Here $\hat f$ viewed as an $\mathbb O_{\mathbb C}$-valued function. %where $x\cdot\xi:=\sum_{i=1}^2\sum_a  x_j^a\xi_j^a.$
Denote \begin{align}
\boldsymbol{\xi}_{ j}:=\mathbf i \sum_a    \bar e_a\xi_j^a\in\mathbb O_{\mathbb C}.
\end{align}
Note that any element of $ \mathbb O_{\mathbb C}$ is complex linear, i.e. $\mathbf i$ commutes   each $e_a.$ By definition (\ref{conj}) of conjugation  in $\mathbb O_{\mathbb C}$,   $\overline{\boldsymbol{\xi}}_{ j}=\mathbf i \sum_a      e_a\xi_j^a.$
\begin{lem}\label{fourier}
Suppose $F\in C_c^1\left(\mathbb O^{2},\mathbb O\right),\xi\in\mathbb O^2$ we have 
\begin{equation}\begin{aligned}
\widehat{ \partial_jF }(\xi)=& {\boldsymbol{\xi}}_j\hat F(\xi),\quad \widehat{ \partial_{\bar j}F }(\xi)= \bar{\boldsymbol{\xi}}_j\hat F(\xi),\\\widehat{ F\overleftarrow{\partial_j} }(\xi)=& \hat F(\xi){\boldsymbol{\xi}}_j,\quad\widehat{ F\overleftarrow{\partial_{\bar j}} }(\xi)= \hat F(\xi)\bar{\boldsymbol{\xi}}_j.
\end{aligned}\end{equation}
In particular,   $\partial_{ j} F=0,$ $\partial_{\bar j} F=0,$ $F\overleftarrow{\partial_{ j}}=0$ and $F\overleftarrow{\partial_{\bar j}}=0 $ are equivalent to   $ {\boldsymbol{\xi}_{ j}} \hat F(\xi)=0,$ $\bar{\boldsymbol{\xi}}_{ j} \hat F(\xi)=0,$ $ \hat F(\xi) { {\boldsymbol{\xi}}_{ j}}=0$ and $\hat F(\xi) {\bar{\boldsymbol{\xi}}_{ j}}=0,$ respectively.
\end{lem}
\begin{proof}
Note that for any $\xi\in\mathbb O^2$\begin{equation*}\begin{aligned}
\widehat{ \partial_jF }(\xi)=&\sum_a  \int_{\mathbb O^2} \bar e_a\partial_{x_j^a}F(x)e^{-\mathbf i\langle x, \xi\rangle}{\rm d}x\\=&\sum_a  \mathbf i \bar e_a\xi_j^a\int_{\mathbb O^2}F(x)e^{-\mathbf i\langle x, \xi\rangle}{\rm d}x=   {\boldsymbol{\xi}}_j\hat F(\xi),
\end{aligned}\end{equation*}
and 
\begin{equation*}\begin{aligned}
\widehat{ F\overleftarrow{\partial_{\bar i}} }(\xi)=&\int_{\mathbb O^2} F(x)\overleftarrow{\partial_{\bar i}}e^{-\mathbf i\langle x, \xi\rangle}{\rm d}x=\sum_a  \int_{\mathbb O^2} \partial_{x_j^a}F(x)  e_a e^{-\mathbf i\langle x, \xi\rangle}{\rm d}x\\=&\sum_a  \int_{\mathbb O^2}F(x)e^{-\mathbf i\langle x, \xi\rangle}{\rm d}x\cdot\mathbf i   e_a\xi_j^a=  \hat F(\xi)\overline{\boldsymbol{\xi}}_j.
\end{aligned}\end{equation*}
 We can prove the remaining identities similarly. 
\end{proof}
Define the symbol $\sigma_j(\xi)$ of differential operator $\mathcal D_j$ at $\xi\in\mathbb O^2\setminus\{0\}$ by $$\sigma_j(\xi)\hat f :=\widehat{\mathcal D_jf} (\xi),$$ for $f\in C_0^\infty\left(\mathbb O^2,\mathcal V_j\right).$  Then by Lemma \ref{fourier}, we have 
\begin{equation}\begin{aligned}\label{sigma01}
\sigma_0( {\xi}){\hat f}=&\left(\begin{matrix} \left|\boldsymbol{\xi}_1\right|^2{\hat f} &\bar{\boldsymbol{\xi}}_1{\hat f} {\boldsymbol{\xi}}_2\\ \bar{\boldsymbol{\xi}}_2{\hat f} {\boldsymbol{\xi}}_1& \left|\boldsymbol{\xi}_2\right|^2{\hat f}\end{matrix}\right),  &&{\rm for}\  {\hat f}\in\mathbb C,\\ \sigma_1(\xi){\hat u}=&\left(\begin{matrix}{\hat u}_{\bar 1 2}  \bar{\boldsymbol{\xi}}_2-  \bar{\boldsymbol{\xi}}_{1}{\hat u}_{\bar 22}\\ {\hat u}_{\bar 21} \bar{\boldsymbol{\xi}}_1-  \bar{\boldsymbol{\xi}}_{2}{\hat u}_{ \bar 1 1}\end{matrix}\right), &&{\rm for}\ {\hat u}=\left(\begin{matrix}{\hat u}_{\bar 11}&{\hat u}_{\bar 12}\\{\hat u}_{ \bar 2 1}&{\hat u}_{\bar 22}\end{matrix}\right)\in\mathcal H_2(\mathbb O_{\mathbb C}),\end{aligned}\end{equation} 
and 
\begin{equation}\begin{aligned}\label{sigma234}\sigma_2(\xi){\hat v}= & \left(\begin{matrix}\left(\bar{\boldsymbol{\xi}}_1 {\boldsymbol{\xi}}_2\right){\hat v}_2 -\left(\bar{\boldsymbol{\xi}}_1\hat{\bar v}_2 \right) \bar{\boldsymbol{\xi}}_2+2{\rm Im}\left(    \hat v_1 { {\boldsymbol{\xi}}_1}\right)\cdot {\bar{\boldsymbol{\xi}}_{1}}
\\
\left(\bar{\boldsymbol{\xi}}_2 {\boldsymbol{\xi}}_1\right) {\hat v}_1-\left(\bar{\boldsymbol{\xi}}_2\hat{\bar v}_1\right) \bar{\boldsymbol{\xi}}_1+ 2{\rm Im}\left(    \hat v_2 { {\boldsymbol{\xi}}_2}\right)\cdot {\bar{\boldsymbol{\xi}}_{2}}\end{matrix}\right),   &&{\rm for}\ {\hat v}=\binom{{\hat v}_1}{{\hat v}_2} \in\mathbb O_{\mathbb C}^2, \\
\sigma_3(\xi){\hat t}=& \left(\begin{matrix} {\rm Re}\left({\boldsymbol{\xi}}_2{\hat t}_2\right) &-\frac{{\bar{\boldsymbol{\xi}}}_1 \hat{\bar t}_2 +{\hat t}_1 {\boldsymbol{\xi}}_2}{2}\\-\frac{{\hat t}_2 {\boldsymbol{\xi}}_1 +\bar{\boldsymbol{\xi}}_2\hat{\bar t}_1}{2}&{\rm Re}\left( {\boldsymbol{\xi}}_1{\hat t}_1 \right)\end{matrix}\right),  &&{\rm for}\ {\hat t}=\binom{\hat t_1} {\hat t_2} 
\in\mathbb O_{\mathbb C}^2,\\\sigma_4(\xi){\hat s}=&\left|\boldsymbol{\xi}_1\right|^2 {\hat s}_{\bar11}+\left|\boldsymbol{\xi}_2\right|^2{\hat s}_{\bar22} +2{\rm Re}\left( {\boldsymbol{\xi}}_2{\hat s}_{\bar21}\bar{\boldsymbol{\xi}}_1 \right), &&{\rm for}\  {\hat s}\in\mathcal H_2\left(\mathbb O_{\mathbb C}\right).
\end{aligned}\end{equation}
Here $$\left|\boldsymbol{\xi}_i\right|^2:= \overline{\boldsymbol{\xi}}_i\boldsymbol{\xi}_i =-\sum_a  \left(\xi_j^a\right)^2\in\mathbb C.$$
 
\begin{prop}\label{complex}
The sequence  {\rm(\ref{co})} with operators given by {\rm (\ref{D0}), (\ref{D1})} and {\rm(\ref{D})-(\ref{D34})} is a   differential complex, i.e. 
\begin{align}
\mathcal D_{j+1}\circ\mathcal D_j=0,
\end{align}for $j=0,\dots,3.$
\end{prop}
\begin{proof} 
 By Lemma \ref{fourier}, we only need to prove their symbols satisfy   $\sigma_{j+1}(\xi)\circ\sigma_j(\xi)=0,$ $j=0,\dots,3.$ For $\hat f\in C_c^\infty(\Omega,\mathbb C),$ we have  
 \begin{equation*}\begin{aligned}
\sigma_1(\xi)\circ\sigma_0(\xi)\hat f=\left(\begin{matrix} \left({\bar{\boldsymbol{\xi}}_1} \hat f { {\boldsymbol{\xi}}_2}\right) {\bar{\boldsymbol{\xi}}_{ 2}}- {\bar{\boldsymbol{\xi}}_{1}} \left|\boldsymbol{\xi}_2\right|^2\hat f \\ \left(\bar{\boldsymbol{\xi}}_2\hat f { {\boldsymbol{\xi}}_1} \right) {\bar{\boldsymbol{\xi}}_1}- \bar{\boldsymbol{\xi}}_{ 2}\left|\boldsymbol{\xi}_{1}\right|^2\hat f \end{matrix}\right)  =\left(\begin{matrix} 0\\0\end{matrix}\right),
\end{aligned}\end{equation*}by definition in  (\ref{sigma01}) and using   alternativity of octonions. Note that   \begin{equation}\begin{aligned}\label{hatbar}\hat{\bar\phi}(\xi) =\int_{\mathbb O^2}\bar\phi(x)e^{-\mathbf i\langle x, \xi\rangle}{\rm d}x=\int_{\mathbb O^2} \overline{\phi(x)e^{-\mathbf i\langle x, \xi\rangle}}{\rm d}x=\bar{\hat\phi}(\xi),\end{aligned}\end{equation} for any $\phi\in L^1\left(\mathbb O^2,\mathbb O\right),$ by complex linearity  of  conjugation  on $\mathbb O_{\mathbb C}$. For $\hat u=\left(\hat u_{\bar ij}\right)\in \mathcal H_2\left(\mathbb O_{\mathbb C}\right),$ we have 
 \begin{equation*}\begin{aligned}
  \left[\sigma_2(\xi)\circ\sigma_1(\xi)\hat u \right]_1   =&\left({\bar{\boldsymbol{\xi}}_1}{ {\boldsymbol{\xi}}_2}\right) \left(\hat u_{\bar21} {\bar{\boldsymbol{\xi}}_{1}}-{ \bar{\boldsymbol{\xi}}_2} \hat u_{\bar11}\right) -\left(\bar{\boldsymbol{\xi}}_{1} \left( {\boldsymbol{\xi}}_{1}\hat u_{\bar 12}-\hat u_{\bar11} {\boldsymbol{\xi}}_2\right)\right) {\bar{\boldsymbol{\xi}}_{2}}\\&+2{\rm Im}\left(\left(\hat u_{\bar12} {\bar{\boldsymbol{\xi}}_{ 2}} -\bar{\boldsymbol{\xi}}_{1} \hat u_{\bar22}\right) { {\boldsymbol{\xi}}_1}\right) \cdot  {\bar{\boldsymbol{\xi}}_{1}} \\=&\bar{\boldsymbol{\xi}}_{1} \left({ {\boldsymbol{\xi}}_2}\hat u_{\bar 21}\right) {\bar{\boldsymbol{\xi}}_{1}}  
 -\left|\boldsymbol{\xi}_2\right|^2{\bar{\boldsymbol{\xi}}_{1}} \hat u_{\bar11} -\left|\boldsymbol{\xi}_1\right|^2\hat u_{\bar 12} {\bar{\boldsymbol{\xi}}_{ 2}}+\left|\boldsymbol{\xi}_2\right|^2 {\bar{\boldsymbol{\xi}}_{1}}\hat u_{\bar11}    +2{\rm Im}\left( \left( \hat u_{\bar12} {\bar{\boldsymbol{\xi}}_{ 2}}\right) { {\boldsymbol{\xi}}_1}\right) \cdot  {\bar{\boldsymbol{\xi}}_{1}} 
\end{aligned}\end{equation*}by definitions in  (\ref{sigma01})-(\ref{sigma234}) and  using  alternativity  of  octonions repeatedly, where $\hat u_{\bar 1 1},\hat u_{\bar 22},$ $\bar{\boldsymbol{\xi}}_{1} \hat u_{\bar22} {\boldsymbol{\xi}}_1\in\mathbb C,$ and $$\left({\bar{\boldsymbol{\xi}}_1}{ {\boldsymbol{\xi}}_2}\right) \left(\hat u_{\bar21} {\bar{\boldsymbol{\xi}}_{1}}\right) ={\bar{\boldsymbol{\xi}}_{1}} \left( {\boldsymbol{\xi}}_2\hat u_{\bar 21}\right) {{\bar{\boldsymbol{\xi}}}_1}, $$  which follows from the   Moufang identity  (\ref{M3}), and \begin{equation*}\begin{aligned}   2{\rm Im}\left( \left(\hat u_{\bar12} {\bar{\boldsymbol{\xi}}_{2}} \right){\boldsymbol{\xi}}_{1}\right) \cdot {\bar{\boldsymbol{\xi}}_{1}}   
=\left|\boldsymbol{\xi}_1\right|^2 \hat u_{\bar 12} {\bar{\boldsymbol{\xi}}_{2}} -\bar{\boldsymbol{\xi}}_{1} \left({ {\boldsymbol{\xi}}_2}\hat u_{\bar 21}\right) {\bar{\boldsymbol{\xi}}_{1}},\end{aligned}\end{equation*} 
   by  using alternativity of octonions. 
     Similarly, \begin{align*}
\left[\sigma_2(\xi)\circ\sigma_1(\xi)\hat u 
\right]_2=0,
 \end{align*} simply  by exchanging indices $1$ and $2$ in the above argument. 
  
  Next for $\hat v= \left(\begin{matrix}\hat v_1\\ \hat v_2 \end{matrix}\right)\in\mathbb O^2_{\mathbb C},$ we have \begin{equation*}\begin{aligned}%\label{takep2}
\left[\sigma_3(\xi)\circ \sigma_2(\xi)\hat v\right]_{\bar 11}=&{\rm Re}\left[ { {\boldsymbol{\xi}}_2} \left(\left( {\bar{\boldsymbol{\xi}}_2}  {\boldsymbol{\xi}}_1 \right)\hat v_1- \left(\bar{\boldsymbol{\xi}}_{2}\hat{\bar{v}}_1\right) {\bar{\boldsymbol{\xi}}_{1}}+2{\rm Im}\left(    \hat v_2 { {\boldsymbol{\xi}}_2}\right)\cdot {\bar{\boldsymbol{\xi}}_{2}} \right) \right]\\=&{\rm Re}\left(\left|{\boldsymbol{\xi}}_2\right|^2 {\boldsymbol{\xi}}_1\hat v_1\right)- {\rm Re}\left(\left|{\boldsymbol{\xi}}_2\right|^2 \bar{\hat v}_1\bar{\boldsymbol{\xi}}_1\right)+2{\rm Re}\left[{\boldsymbol{\xi}}_{2}\left({\rm Im}\left(    \hat v_2 { {\boldsymbol{\xi}}_2}\right)\cdot {\bar{\boldsymbol{\xi}}_{2}}\right)\right]=0  
 \end{aligned}\end{equation*} 
by using Lemma \ref{asso} (i)-(iii) and $${\rm Re}\left[{\boldsymbol{\xi}}_{2}\left({\rm Im}\left(    \hat v_2 { {\boldsymbol{\xi}}_2}\right)\cdot {\bar{\boldsymbol{\xi}}_{2}}\right)\right]=\left| {\boldsymbol{\xi}}_{2}\right|^2{\rm Re} \left({\rm Im}\left(    \hat v_2 { {\boldsymbol{\xi}}_2}\right)  \right)=0. $$  
Similarly \begin{equation*}\begin{aligned}
\left[\sigma_3(\xi)\circ \sigma_2(\xi)\hat v\right]_{\bar 22}=0,\end{aligned}\end{equation*}
and \begin{equation*}\begin{aligned}
2\left[\sigma_3(\xi)\circ \sigma_2(\xi)\hat v\right]_{\bar 21}=&- \left[\left( \bar{\boldsymbol{\xi}}_{2}  {\boldsymbol{\xi}}_1 \right)\hat v_1- \left(\bar{\boldsymbol{\xi}}_{2}\hat{\bar v}_1\right) {\bar{\boldsymbol{\xi}}_{1}}+ \left(\hat v_2 { {\boldsymbol{\xi}}_2}\right) {\bar{\boldsymbol{\xi}}_{2}}-\bar{\boldsymbol{\xi}}_{2}\hat{\bar v}_2 {\bar{\boldsymbol{\xi}}_{2}}\right] { {\boldsymbol{\xi}}_1} \\&- {\bar{\boldsymbol{\xi}}_{2}}\left[\hat{\bar v}_2 \left( \bar{\boldsymbol{\xi}}_{2} { {\boldsymbol{\xi}}_1} \right) -{ {\boldsymbol{\xi}}_{2}} \left(\hat v_2 { {\boldsymbol{\xi}}_1} \right) +  {\boldsymbol{\xi}}_1\left(\bar{\boldsymbol{\xi}}_{1}\hat{\bar v}_1\right)-{ {\boldsymbol{\xi}}_1}\hat v_1 { {\boldsymbol{\xi}}_1}\right]\\ =& \left|\boldsymbol{\xi}_1\right|^2 \bar{\boldsymbol{\xi}}_{2}\hat{\bar v}_1  -\left|\boldsymbol{\xi}_2\right|^2  \hat v_2 { {\boldsymbol{\xi}}_1} + \left|\boldsymbol{\xi}_2\right|^2\hat v_2 { {\boldsymbol{\xi}}_1} -\left|\boldsymbol{\xi}_1\right|^2 \bar{\boldsymbol{\xi}}_{2}\hat{\bar v}_1 = 0, 
\end{aligned}\end{equation*}by using the weak form of the associativity   repeatedly, and \begin{equation}\begin{aligned}\left(\left(\bar{\boldsymbol{\xi}}_{2}{ {\boldsymbol{\xi}}_1}\right)\hat v_1\right) { {\boldsymbol{\xi}}_1}=&\bar{\boldsymbol{\xi}}_{2} \left({ {\boldsymbol{\xi}}_1}\hat v_1 { {\boldsymbol{\xi}}_1}\right), \\  \bar{\boldsymbol{\xi}}_{2}\left(\hat{\bar v}_2\left(  {\bar{\boldsymbol{\xi}}_{2}} { {\boldsymbol{\xi}}_1} \right)\right) =&\left(\bar{\boldsymbol{\xi}}_{2}\hat{\bar v}_2  {\bar{\boldsymbol{\xi}}_{2}}\right) { {\boldsymbol{\xi}}_1},\end{aligned}\end{equation}  by using the   Moufang identities (\ref{M1})-(\ref{M2}). 

Finally, for $\hat t=\left(\begin{matrix}  \hat t_1\\  \hat t_2 \end{matrix}\right)\in\mathbb O^2_{\mathbb C},$ we have \begin{equation*}\begin{aligned}
\sigma_4(\xi)\circ\sigma_3(\xi)\hat t =
&\left|\boldsymbol{\xi}_1\right|^2 {\rm Re}\left( {\boldsymbol{\xi}}_2  \hat t_2 \right) +\left|\boldsymbol{\xi}_2\right|^2  {\rm Re}\left({ {\boldsymbol{\xi}}_1}\hat t_1  \right)   - {\rm Re}\left(\left[{ {\boldsymbol{\xi}}_2} \left(\hat t_2 { {\boldsymbol{\xi}}_1} +{\bar{\boldsymbol{\xi}}_2}\hat{\bar t}_1 \right) \right] {\bar{\boldsymbol{\xi}}}_{1}\right)\\ =&\left|\boldsymbol{\xi}_1\right|^2 {\rm Re} \left( {\boldsymbol{\xi}}_2  \hat t_2  \right) +\left|\boldsymbol{\xi}_2\right|^2 {\rm Re} \left({ {\boldsymbol{\xi}}_1}\hat t_1 \right)   - {\rm Re}\left(\boldsymbol{\xi}_2\hat t_2\left|\boldsymbol{\xi}_1\right|^2+\left|\boldsymbol{\xi}_2 \right|^2\hat{\bar t}_1\overline{\boldsymbol{\xi}}_1\right) =0, 
\end{aligned}\end{equation*}by using   Lemma \ref{asso} (iii). 
The theorem is proved. 
\end{proof}
\begin{rem}The octonionic Hessian complex 
 {\rm(\ref{co})} is a self dual complex.
\end{rem}
To find the compatibility condition of operator $\mathcal D_1,$ note that \begin{equation}\begin{aligned}\label{1122}
\partial_2v_2=&\partial_2\left(u_{\bar21} \overleftarrow{\partial_{\bar1}}-\partial_{\bar2}u_{\bar11}\right) =\partial_2\left(u_{\bar21} \overleftarrow{\partial_{\bar1}}\right)-\Delta_2u_{\bar11}, \\\partial_{\bar1}\bar v_1=&\partial_{\bar1}\overline{\left(u_{\bar12} \overleftarrow{\partial_{\bar2}}-\partial_{\bar 1}u_{\bar22}\right)}=\partial_{\bar1}\left(\partial_2u_{\bar21} \right)-\Delta_1u_{\bar22}
\end{aligned}\end{equation}by $u_{\bar11},u_{\bar22}$ scalar.
Note that \begin{align}\label{guess}\partial_{\bar1}\left(\partial_2 \left(u_{\bar21} \overleftarrow{\partial_{\bar1}}\right)\right)=\left( \partial_{\bar1}\left(\partial_2 u_{\bar21} \right)\right) \overleftarrow{\partial_{\bar1}},\end{align} if associativity holds. 
Then, taking imaginary to delete real parts of  (\ref{1122}), we may guess \begin{align}\label{condi}
\partial_{\bar1}{\rm Im}\left(\partial_2v_2\right)={\rm Im}\left({\partial_{\bar 1}}\bar v_1\right)\partial_{\bar1}
\end{align} to be  a compatibility condition of operator $\mathcal D_1,$ and  exchange indices $1$ and $2$ to get the other one.  Unfortunately, the associativity  (\ref{guess})  does not holds.  We can only use Moufang identities and alternativity as  our main tool to overcome the difficulty,  
 but the left hand side of (\ref{guess}) is not a product appearing Moufang identities in  Theorem B. Therefore  we modify (\ref{condi}) to  the operator  $$\left( {\partial_{\bar 1}} {\partial_2} \right)v_2- \left({\partial_{\bar 1}}\bar v_2\right)\overleftarrow{\partial_{\bar 2}}=2{\rm Im}\left( {\partial_{\bar 1}}\bar v_1\right) \overleftarrow{\partial_{\bar 1}},$$ so that we can use Moufang identities and alternativity to prove it to be the compatibility condition of $\mathcal D_1u=0.$  This is operator $\mathcal D_2$ in (\ref{D}). We made this   modification so  that  we can prove their symbol sequence is exact, and so is the indices  $1$ and $2$ exchanged. These two conditions are sufficient.

\subsection{The ellipticity of the octonionic Hessian  complex}

\begin{prop}\label{pell}
The  complex {\rm(\ref{co})}   is elliptic.
\end{prop}
\begin{proof}
In order to   prove  {\rm(\ref{co})} to be    elliptic, we need to prove  the exactness of the symbol sequence (\ref{simco}).
  It follows from  $\mathcal D_{l+1}\circ\mathcal D_l=0$ that $\sigma_{l+1}(\xi)\circ\sigma_l(\xi)=0,$ for   ${\xi}\in\mathbb O^{2}\setminus\{\mathbf 0\},$ i.e. ${\rm Im}\ \sigma_l(\xi) {\subset}\ker\sigma_{l+1}(\xi),$ for $l=0,\dots,3.$ So  we only need to prove that the injectivity of  $\sigma_0(\xi)$, $\ker\sigma_{l+1}(\xi)\subset{\rm Im}\ \sigma_{l}(\xi),$   and the surjectivity of  $\sigma_4(\xi)$.

By  definition (\ref{sigma01}),  
for any ${\hat f}\in\ker \sigma_0(\xi),$ we have $$\left|\boldsymbol{\xi}_1\right|^2{\hat f} =\left|\boldsymbol{\xi}_2\right|^2{\hat f}=0.$$ Since 
$\left|\boldsymbol{\xi}_1\right|^2$ and $\left|\boldsymbol{\xi}_2\right|^2$ can not vanish simultaneously, we get    ${\hat f}=0.$ Hence $\sigma_0(\xi)$ is injective.

(i) Now consider $\ker \sigma_1( {\xi})$.  Let  ${\hat u} \in \ker \sigma_1(\xi)$ with $\left|\boldsymbol{\xi}_1\right|\neq0,$ and let
$
{\hat f}=\frac{{\hat u}_{\bar 11}}{\left|\boldsymbol{\xi}_1\right|^2}\in\mathbb C.
$
We claim that $\sigma_0({\xi}){\hat f}={\hat u}.$ 
By definition in (\ref{sigma01}),  we have
\begin{align}\label{s1}{\hat u}_{\bar 1 2}  \bar{\boldsymbol{\xi}}_2-  \bar{\boldsymbol{\xi}}_{1}{\hat u}_{ \bar 2 2}=0={\hat u}_{\bar 2 1} \bar{\boldsymbol{\xi}}_1-  \bar{\boldsymbol{\xi}}_{2}{\hat u}_{\bar 1 1}.\end{align}
Then we get  \begin{equation*}\begin{aligned} \left(\sigma_0(\xi){\hat f}\right)_{ \bar 1 1}=&\left|\boldsymbol{\xi}_1\right|^2\cdot\frac{{\hat u}_{ \bar 1 1} }{\left|\boldsymbol{\xi}_1\right|^2} ={\hat u}_{\bar 1 1},\\ \left(\sigma_0(\xi){\hat f}\right)_{ \bar 1 2}=&\frac{\bar{\boldsymbol{\xi}}_1{\hat u}_{ \bar 1 1} {\boldsymbol{\xi}}_2}{\left|\boldsymbol{\xi}_1\right|^2} =\frac{\bar{\boldsymbol{\xi}}_1\left( {\boldsymbol{\xi}}_1{\hat u}_{\bar 1 2}\right)}{\left|\boldsymbol{\xi}_1\right|^2}={\hat u}_{1\bar 2},\\\left(\sigma_0(\xi){\hat f}\right)_{ \bar 2 1}=&\frac{\bar{\boldsymbol{\xi}}_2{\hat u}_{ \bar 1 1} {\boldsymbol{\xi}}_1}{\left| \boldsymbol{\xi}_1\right|^2}=\frac{{\hat u}_{ \bar 2 1}\bar{\boldsymbol{\xi}}_1 {\boldsymbol{\xi}}_1} {\left|\boldsymbol{\xi}_1\right|^2}={\hat u}_{\bar 2 1},\end{aligned}\end{equation*}
by using (\ref{s1}) and $$\bar{\hat u}_{\bar ij}=\hat{  u}_{\bar ji},\quad i,j=1,2,$$ by (\ref{hatbar}). By multiplying $\boldsymbol{\xi}_1$ and $\boldsymbol{\xi}_2$ from the  left to the first and second identity of  (\ref{s1}), respectively,  we get \begin{align*}
 {\boldsymbol{\xi}}_1 \left({\hat u}_{\bar12}\bar{\boldsymbol{\xi}}_2\right) =\left|\boldsymbol{\xi}_1\right|^2{\hat u}_{\bar22},\quad {\boldsymbol{\xi}}_2 \left({\hat u}_{\bar21} \bar{\boldsymbol{\xi}}_1\right) =\left|\boldsymbol{\xi}_2\right|^2{\hat u}_{\bar11},
\end{align*}
which are both complex.  
On the other hand, since ${\rm Re} \left[{\boldsymbol{\xi}}_1 \left({\hat u}_{\bar12}\bar{\boldsymbol{\xi}}_2\right)  \right]={\rm Re}\left[{\boldsymbol{\xi}}_2 \left({\hat u}_{\bar21} \bar{\boldsymbol{\xi}}_1\right)\right]$ by Lemma \ref{asso} (iii),   we get 
$\left|\boldsymbol{\xi}_2\right|^2{\hat u}_{\bar 1 1}= \left|\boldsymbol{\xi}_1\right|^2{\hat u}_{\bar 2 2},$ and so \begin{align*}
\left(\sigma_0(\xi){\hat f}\right)_{ \bar 2 2}=\frac{\left|\boldsymbol{\xi}_2\right|^2{\hat u}_{ \bar 1 1}}{\left|\boldsymbol{\xi}_1\right|^2}={\hat u}_{\bar 2 2}.
\end{align*}
 Thus $\sigma_0(\xi){\hat f}={\hat u}.$ If $\left|\boldsymbol{\xi}_1\right|=0,$ we must have $\left|\boldsymbol{\xi}_2\right|\neq0,$ let
$
{\hat f}=\frac{{\hat u}_{\bar 22}}{|\boldsymbol{\xi}_2|^2}\in\mathbb C.
$
We still have   $\sigma_0(\xi){\hat f}={\hat u}$. So $\ker\sigma_1(\xi)\subset{\rm Im}\ \sigma_0(\xi).$

(ii) Next consider $\ker \sigma_2(\xi)$.  
By definition, for any   ${\hat v}=\left(\begin{matrix}{\hat v}_1\\{\hat v}_2\end{matrix}\right)\in \ker \sigma_2(\xi),$ we have  \begin{equation}\begin{aligned}\label{id2}\left(\bar{\boldsymbol{\xi}}_1 {\boldsymbol{\xi}}_2\right){\hat v}_2 -\left(\bar{\boldsymbol{\xi}}_1\hat{\bar v}_2 \right) \bar{\boldsymbol{\xi}}_2+{\hat v}_1 \left|\boldsymbol{\xi}_1\right|^2-\bar{\boldsymbol{\xi}}_1\hat{\bar v}_1\bar{\boldsymbol{\xi}}_1=0,\\\left(\bar{\boldsymbol{\xi}}_2 {\boldsymbol{\xi}}_1\right) {\hat v}_1-\left(\bar{\boldsymbol{\xi}}_2\hat{\bar v}_1\right) \bar{\boldsymbol{\xi}}_1+{\hat v}_2\left| \boldsymbol{\xi}_2\right|^2-\bar{\boldsymbol{\xi}}_2\hat{\bar v}_2\bar{\boldsymbol{\xi}}_2=0.\end{aligned}\end{equation}    
  By multiplying $ {\boldsymbol{\xi}}_1$ and $ {\boldsymbol{\xi}}_2$ from the  right to   the  first and second identities  of (\ref{id2}), respectively,  we get \begin{equation}\begin{aligned}\label{bar}
\left(\left(\bar{\boldsymbol{\xi}}_1 {\boldsymbol{\xi}}_2\right){\hat v}_2\right)  {\boldsymbol{\xi}}_1 -\left(\left(\bar{\boldsymbol{\xi}}_1 \hat{\bar v}_2\right)\bar{\boldsymbol{\xi}}_2\right) {\boldsymbol{\xi}}_1 =\left|\boldsymbol{\xi}_1\right|^2\left( \bar{\boldsymbol{\xi}}_1\hat{\bar v}_1   - {\hat v}_1 {\boldsymbol{\xi}}_1 \right)\in{\rm Im}\ \mathbb O_{\mathbb C},\\\left(\left(\bar{\boldsymbol{\xi}}_2 {\boldsymbol{\xi}}_1\right){\hat v}_1\right)  {\boldsymbol{\xi}}_2 -\left(\left(\bar{\boldsymbol{\xi}}_2\hat{\bar v}_1\right)\bar{\boldsymbol{\xi}}_1\right) {\boldsymbol{\xi}}_2 =\left|\boldsymbol{\xi}_2\right|^2\left( \bar{\boldsymbol{\xi}}_2\hat{\bar v}_2   - {\hat v}_2 {\boldsymbol{\xi}}_2 \right)\in{\rm Im}\ \mathbb O_{\mathbb C}.
 \end{aligned}\end{equation}
For $\left|\boldsymbol{\xi}_1\right|\neq0,\left| \boldsymbol{\xi}_2\right|\neq0,$ let \begin{equation}\begin{aligned}\label{xi}{\hat u}_{\bar 1 1}:=&-\frac{ {\boldsymbol{\xi}}_2{\hat v}_2+ \hat{\bar v}_2\bar{\boldsymbol{\xi}}_2 }{2\left|\boldsymbol{\xi}_2\right|^2},\\{\hat u}_{\bar 12}:=&\frac{ \left| \boldsymbol{\xi}_1\right|^2  {\hat v}_1 {\boldsymbol{\xi}}_2 -\left(\bar{\boldsymbol{\xi}}_1\hat{\bar v}_1  \bar{\boldsymbol{\xi}}_1\right) {\boldsymbol{\xi}}_2} {2\left|\boldsymbol{\xi}_1\right|^2 \left|\boldsymbol{\xi}_2\right|^2},\\{\hat u}_{\bar 2 1}:=&\frac{  \left| \boldsymbol{\xi}_2\right|^2 {\hat v}_2 {\boldsymbol{\xi}}_1  - \left(\bar{\boldsymbol{\xi}}_2\hat{\bar v}_2  \bar{\boldsymbol{\xi}}_2\right) {\boldsymbol{\xi}}_1 }{2\left|\boldsymbol{\xi}_1\right|^2 \left|\boldsymbol{\xi}_2\right|^2},\\{\hat u}_{\bar 2 2}:=&-\frac{ {\boldsymbol{\xi}}_1  {\hat v}_1+\hat{\bar v}_1 \bar{\boldsymbol{\xi}}_1 }{2\left|\boldsymbol{\xi}_1\right|^2}. \end{aligned}\end{equation}

To check ${\hat u}=\left( {\hat u}_{\bar ij} \right)\in\mathcal H_2(\mathbb O_{\mathbb C}),$ note that by multiplying $\bar{\boldsymbol{\xi}}_1$ and then $ {\boldsymbol{\xi}}_2$ from the  right to the first identity of (\ref{bar}), we get \begin{equation*}\begin{aligned}
\left|\boldsymbol{\xi}_1\right|^2 \left[\left(\left(\bar{\boldsymbol{\xi}}_1 {\boldsymbol{\xi}}_2\right){\hat v}_2\right)  {\boldsymbol{\xi}}_2 -\left(\left(\bar{\boldsymbol{\xi}}_1\hat{\bar v}_2\right) \bar{\boldsymbol{\xi}}_2\right) {\boldsymbol{\xi}}_2\right] =\left|\boldsymbol{\xi}_1\right|^2\left[\left( \bar{\boldsymbol{\xi}}_1\hat{\bar v}_1   - {\hat v}_1 {\boldsymbol{\xi}}_1 \right) \bar{\boldsymbol{\xi}}_1\right] {\boldsymbol{\xi}}_2.
\end{aligned}\end{equation*} Then,  
 we have \begin{equation}\begin{aligned}\label{bar1}
\bar{\boldsymbol{\xi}}_1 \left( {\boldsymbol{\xi}}_2 {\hat v}_2 {\boldsymbol{\xi}}_2 \right) -\left|\boldsymbol{\xi}_2\right|^2 \bar{\boldsymbol{\xi}}_1\hat{\bar v}_2  =   \left(\bar{\boldsymbol{\xi}}_1\hat{\bar v}_1 \bar{\boldsymbol{\xi}}_1\right)  {\boldsymbol{\xi}}_2  - \left|\boldsymbol{\xi}_1\right|^2{\hat v}_1   {\boldsymbol{\xi}}_2,
\end{aligned}\end{equation} by $\left|{\boldsymbol{\xi}}_1\right|\neq0$ and  using the alternativity  and  the   Moufang identity (\ref{M2}).
(\ref{bar1}) implies   ${\hat u}_{\bar12}=\overline{{\hat u}_{\bar 21}}$ in (\ref{xi}).  
So $u$ given in (\ref{xi}) belongs to $\mathcal H_2\left(\mathbb O_{\mathbb C}\right).$ It is direct to see that  
\begin{equation*}\begin{aligned}\label{sigma1} \left(\sigma_1(\xi){\hat u}\right)_{1}=&{\hat u}_{ \bar 12} \bar{\boldsymbol{\xi}}_2- \bar{\boldsymbol{\xi}}_1{\hat u}_{\bar 2 2}= \frac{ \left| \boldsymbol{\xi}_1\right|^2  {\hat v}_1 {\boldsymbol{\xi}}_2 -\left(\bar{\boldsymbol{\xi}}_1\hat{\bar v}_1  \bar{\boldsymbol{\xi}}_1\right) {\boldsymbol{\xi}}_2}{2 \left|\boldsymbol{\xi}_1\right|^2\left|\boldsymbol{\xi}_2 \right|^2}\cdot \bar{\boldsymbol{\xi}}_2 + \bar{\boldsymbol{\xi}}_1\cdot\frac{ {\boldsymbol{\xi}}_1  {\hat v}_1+\hat{\bar v}_1 \bar{\boldsymbol{\xi}}_1 }{2\left|\boldsymbol{\xi}_1\right|^2}   =  {\hat v}_1, \end{aligned}\end{equation*} 
  by definition in  (\ref{xi}) and  alternativity, and     
 \begin{equation*}\begin{aligned} \left(\sigma_1(\xi){\hat u}\right)_2=& {\hat u}_{\bar 21}  \bar{\boldsymbol{\xi}}_1- \bar{\boldsymbol{\xi}}_2{\hat u}_{\bar 11}=  \frac{  \left| \boldsymbol{\xi}_2\right|^2 {\hat v}_2 {\boldsymbol{\xi}}_1  - \left(\bar{\boldsymbol{\xi}}_2\hat{\bar v}_2  \bar{\boldsymbol{\xi}}_2\right) {\boldsymbol{\xi}}_1 }{2\left|\boldsymbol{\xi}_1\right|^2 \left|\boldsymbol{\xi}_2\right|^2}\cdot \bar{\boldsymbol{\xi}}_1 +\bar{\boldsymbol{\xi}}_2\cdot\frac{ {\boldsymbol{\xi}}_2{\hat v}_2+ \hat{\bar v}_2\bar{\boldsymbol{\xi}}_2 }{2\left|\boldsymbol{\xi}_2\right|^2}   =  {\hat v}_2. \end{aligned}\end{equation*}
  If $\left|\boldsymbol{\xi}_1\right|=0,$ we must  have $\left|\boldsymbol{\xi}_2\right|\neq0$ and ${\hat v}_2 {\boldsymbol{\xi}}_2=\bar{\boldsymbol{\xi}}_2\hat{\bar v}_2$ by the second equation in  (\ref{bar}), and so  ${\hat v}_2 {\boldsymbol{\xi}}_2\in\mathbb C.$ Let\begin{equation*}\begin{aligned}
  {\hat u}_{\bar 11}=-\frac{{\hat v}_2 {\boldsymbol{\xi}}_2}{ \left|\boldsymbol{\xi}_2\right|^2},\quad {\hat u}_{\bar 12}=\frac{{\hat v}_1 {\boldsymbol{\xi}}_2}{ \left|\boldsymbol{\xi}_2\right|^2} =\overline{{\hat u}_{\bar 21}},\quad {\hat u}_{\bar 22}=0. 
  \end{aligned}\end{equation*}
Then,  we have \begin{equation*}\begin{aligned} \left(\sigma_1(\xi){\hat u}\right)_{1}=&{\hat u}_{ \bar 12} \bar{\boldsymbol{\xi}}_2 = \left(\frac{{\hat v}_1 {\boldsymbol{\xi}}_2} {\left|\boldsymbol{\xi}_2\right|^2}\right)\bar{\boldsymbol{\xi}}_2= {\hat v}_1,\\ \left(\sigma_1(\xi){\hat u}\right)_2=&  - \bar{\boldsymbol{\xi}}_2{\hat u}_{\bar 11}=-\bar{\boldsymbol{\xi}}_2\left(-\frac{{\hat v}_2 {\boldsymbol{\xi}}_2}{\left|\boldsymbol{\xi}_2\right|^2} \right)= \left( \frac{{\hat v}_2 {\boldsymbol{\xi}}_2} {\left|\boldsymbol{\xi}_2\right|^2}\right)\bar{\boldsymbol{\xi}}_2 ={\hat v}_2, \end{aligned}\end{equation*}where the last second holds by ${\hat v}_2 {\boldsymbol{\xi}}_2\in\mathbb C$ in this case. Similarly we can prove the case of  $\left|\boldsymbol{\xi}_2\right|=0$.
Thus $\sigma_1(\xi){\hat u}={\hat v}$ and so $\ker\sigma_2(\xi)\subset{\rm Im}\ \sigma_1(\xi).$

  (iii) Next consider $\ker \sigma_3(\xi)$. By definition in (\ref{sigma234}),    for any   ${\hat t}=\left(\begin{matrix}{\hat t}_1\\{\hat t}_2\end{matrix}\right)\in \ker \sigma_3(\xi),$ we have  \begin{equation}\label{zeta}\begin{matrix} {\boldsymbol{\xi}}_2{\hat t}_2+  \hat{\bar t}_2  \bar{\boldsymbol{\xi}}_2=0, \\{\boldsymbol{\xi}}_1{\hat t}_1+\hat{\bar t}_1\bar{\boldsymbol{\xi}}_1=0,\\{\hat t}_1 {\boldsymbol{\xi}}_2 +\bar{\boldsymbol{\xi}}_1\hat{\bar t}_2=0. \end{matrix}
\end{equation}
If take $${\hat v}_1 =\frac{ {\hat t}_1}{2\left( \left|\boldsymbol{\xi}_1\right|^2 +\left|\boldsymbol{\xi}_2\right|^2\right)},\quad {\hat v}_2=\frac{  {\hat t}_2}{2\left(\left|\boldsymbol{\xi}_1\right|^2 +\left|\boldsymbol{\xi}_2\right|^2\right)},$$    we have 
   \begin{equation*}\begin{aligned} \left(\sigma_2(\xi){\hat v}\right)_{1} =& \left(\bar{\boldsymbol{\xi}}_1 {\boldsymbol{\xi}}_2\right) {\hat v}_2-\left(\bar{\boldsymbol{\xi}}_1\hat{\bar v}_2\right) \bar{\boldsymbol{\xi}}_2+{\hat v}_1\left|\boldsymbol{\xi}_1\right|^2 -\bar{\boldsymbol{\xi}}_1\hat{\bar v}_1\bar{\boldsymbol{\xi}}_1\\ =& \frac{1}{2\left(\left|\boldsymbol{\xi}_1\right|^2 +\left|\boldsymbol{\xi}_2\right|^2\right)} \left(\left(\bar{\boldsymbol{\xi}}_1 {\boldsymbol{\xi}}_2\right) {\hat t}_2 -\left(\bar{\boldsymbol{\xi}}_1\hat{\bar t}_2\right) \bar{\boldsymbol{\xi}}_2+{\hat t}_1\left|\boldsymbol{\xi}_1\right|^2 -\bar{\boldsymbol{\xi}}_1\hat{\bar t}_1\bar{\boldsymbol{\xi}}_1\right)\\ =&\frac{1}{2\left(\left|\boldsymbol{\xi}_1\right|^2 +\left|\boldsymbol{\xi}_2\right|^2\right)} \left(\left(\bar{\boldsymbol{\xi}}_1 {\boldsymbol{\xi}}_2\right) {\hat t}_2 +\left({\hat t}_1 {\boldsymbol{\xi}}_2\right) \bar{\boldsymbol{\xi}}_2+{\hat t}_1\left|\boldsymbol{\xi}_1\right|^2 +\bar{\boldsymbol{\xi}}_1\left( {\boldsymbol{\xi}}_1 {\hat t}_1\right)\right)  = {\hat t}_1, \end{aligned}\end{equation*}
  by using (\ref{zeta}),   alternativity of octonions and  \begin{align}\label{cl} \left(\bar{\boldsymbol{\xi}}_1 {\boldsymbol{\xi}}_2\right){\hat t}_2 =\left|\boldsymbol{\xi}_2\right|^2{\hat t}_1.\end{align} 
This is because   \begin{equation*}\begin{aligned}
 \left[\left(\bar{\boldsymbol{\xi}}_1 {\boldsymbol{\xi}}_2\right) {\hat t}_2\right] {\boldsymbol{\xi}}_2= \bar{\boldsymbol{\xi}}_1\left[ {\boldsymbol{\xi}}_2 {\hat t}_2 {\boldsymbol{\xi}}_2\right]=-\bar{\boldsymbol{\xi}}_1 \hat{\bar t}_2\left|\boldsymbol{\xi}_2\right|^2 ={\hat t}_1 \left|\boldsymbol{\xi}_2\right|^2{\boldsymbol{\xi}}_2,
\end{aligned}\end{equation*}
where the first  identity  holds by using the    Moufang identity (\ref{M2}), while the second and third identities holds by (\ref{zeta}). Then (\ref{cl}) follows if  $\left|{\boldsymbol{\xi}}_2\right|\neq0,$ while   (\ref{cl}) holds obviously if $\left|{\boldsymbol{\xi}}_2\right|=0$. 
Similarly, we have \begin{equation*}\begin{aligned} \left(\sigma_2(\xi){\hat v}\right)_{2}={\hat t}_2, \end{aligned}\end{equation*} by exchanging indices $1$ and $2.$  Thus $\sigma_2(\xi){\hat v}={\hat t}$ and so $\ker\sigma_3(\xi)\subset{\rm Im}\ \sigma_2(\xi).$
  
(iv) Next consider $\ker \sigma_4(\xi)$. For any   ${\hat s}=\left({\hat s}_{\bar ij}\right)\in \ker \sigma_4(\xi),$ where ${\hat s}_{\bar11},{\hat s}_{\bar22}\in\mathbb C,{\hat s}_{\bar12}=\overline{{\hat s}_{\bar21}}\in\mathbb O_{\mathbb C},$ we have  \begin{equation}\label{s3}
 \sigma_4(\xi) {\hat s} = \left|\boldsymbol{\xi}_1\right|^2 {\hat s}_{\bar11}+\left|\boldsymbol{\xi}_2\right|^2{\hat s}_{\bar22} +2{\rm Re}\left( {\boldsymbol{\xi}}_2 {\hat s}_{\bar21}\bar{\boldsymbol{\xi}}_1\right)=0.
  \end{equation}
For $\left|\boldsymbol{\xi}_2\right|\neq0,$ if let $${\hat t}_1=-\frac{2 {\hat s}_{\bar12} \bar{\boldsymbol{\xi}}_2 }{\left|\boldsymbol{\xi}_2\right|^2}- \frac{ \bar{\boldsymbol{\xi}}_1{\hat s}_{\bar11} }{ \left|\boldsymbol{\xi}_2\right|^2}, \quad {\hat t}_2=\frac{ \bar{\boldsymbol{\xi}}_2 {\hat s}_{\bar11}}{\left|\boldsymbol{\xi}_2\right|^2},$$ we have 
\begin{equation*}\begin{aligned}
\left(\sigma_3(\xi){\hat t}\right)_{\bar11} =& {\rm Re}\left( {\boldsymbol{\xi}}_2{\hat t}_2 \right) 
={\hat s}_{\bar11},
\\\left(\sigma_3(\xi) {\hat t}\right)_{\bar22} =&{\rm Re}\left( {\boldsymbol{\xi}}_1{\hat t}_1 \right)=\frac{-2{\rm Re}\left(\boldsymbol{\xi}_1 {\hat s}_{\bar12} \bar{\boldsymbol{\xi}}_2\right)-  \left|\boldsymbol{\xi}_1\right|^2 {\hat s}_{\bar11}} {\left|\boldsymbol{\xi}_2\right|^2} ={\hat s}_{\bar22},
\end{aligned}\end{equation*}    by using (\ref{s3}),  
  and \begin{equation*}\begin{aligned}
\left(\sigma_3(\xi){\hat t}\right)_{\bar21} =&\frac{-{\hat t}_2 {\boldsymbol{\xi}}_1 - \bar{\boldsymbol{\xi}}_2\hat{\bar t}_1}{2}=- \frac{\bar{\boldsymbol{\xi}}_2{\hat s}_{\bar11} {\boldsymbol{\xi}}_1}{2\left|\boldsymbol{\xi}_2\right|^2} +\frac{ \bar{\boldsymbol{\xi}}_2\boldsymbol{\xi}_2{\hat s}_{\bar21}} {\left|\boldsymbol{\xi}_2\right|^2} +\frac{\bar{\boldsymbol{\xi}}_2{\hat s}_{\bar11} {\boldsymbol{\xi}}_1}{2\left|\boldsymbol{\xi}_2\right|^2} ={\hat s}_{\bar21},
 \end{aligned}\end{equation*} by the fact that ${\hat s}_{\bar11}\in\mathbb C.$
 If $\left|\boldsymbol{\xi}_2\right|=0,$ we must have $\left|\boldsymbol{\xi}_1\right|\neq0.$ Let $${\hat t}_1=\frac{\bar{\boldsymbol{\xi}}_1 {\hat s}_{\bar22}}{ \left|\boldsymbol{\xi}_1\right|^2},\quad {\hat t}_2=-\frac{2 {\hat s}_{\bar21}\bar{\boldsymbol{\xi}}_1  }{\left|\boldsymbol{\xi}_1\right|^2}- \frac{\bar{\boldsymbol{\xi}}_2{\hat s}_{\bar22} }{ \left|\boldsymbol{\xi}_1\right|^2}.$$ Then we have 
 \begin{equation*}\begin{aligned} \left(\sigma_3(\xi){\hat t}\right)_{\bar11} =& {\rm Re}\left({\boldsymbol{\xi}}_2{\hat t}_2\right) =\frac{-2{\rm Re}\left( {\boldsymbol{\xi}}_2 {\hat s}_{\bar21}\bar{\boldsymbol{\xi}}_1\right)-{\hat s}_{\bar22} \left|\boldsymbol{\xi}_2\right|^2} {\left|\boldsymbol{\xi}_1\right|^2}={\hat s}_{\bar11},\\
 \left(\sigma_3(\xi){\hat t}\right)_{\bar2 2} =& {\rm Re}\left( {\boldsymbol{\xi}}_1{\hat t}_1\right)={\hat s}_{\bar22},\\ \left(\sigma_3(\xi){\hat t}\right)_{\bar21} =&\frac{-{\hat t}_2 {\boldsymbol{\xi}}_1 - \bar{\boldsymbol{\xi}}_2\hat{\bar t}_1}{2}=  \frac{{\hat s}_{\bar21}\bar{\boldsymbol{\xi}}_1 {\boldsymbol{\xi}}_1}{ \left|\boldsymbol{\xi}_1\right|^2} +\frac{\bar{\boldsymbol{\xi}}_2{\hat s}_{\bar22} {\boldsymbol{\xi}}_1}{2\left|\boldsymbol{\xi}_1\right|^2} - \frac{\bar{\boldsymbol{\xi}}_2{\hat s}_{\bar22} {\boldsymbol{\xi}}_1}{2\left|\boldsymbol{\xi}_1\right|^2} ={\hat s}_{\bar21}.
 \end{aligned}\end{equation*}
Thus, $\sigma_3(\xi){\hat t}={\hat s}$ and so $\ker\sigma_4(\xi)\subset{\rm Im}\ \sigma_3(\xi).$ 
 
(v) At last we   consider ${\rm Im}\ \sigma_4(\xi)$. For any   $\hat \tau\in  \mathbb C $ and $\left|\boldsymbol{\xi}_1\right|\neq0, \left|\boldsymbol{\xi}_2\right|\neq0,$ let ${\hat s}=\left( {\hat s}_{\bar ij} \right)\in \mathcal H_2\left(\mathbb O_{\mathbb C}\right)$ be  given by
\begin{equation*}\begin{aligned}
{\hat s}_{\bar11}=&{\hat s}_{\bar22} =\frac{\hat \tau}{2\left(|\boldsymbol{\xi}_1|^2 +|\boldsymbol{\xi}_2|^2\right)}, \\{\hat s}_{\bar21}=& \frac{\bar{\boldsymbol{\xi}}_{2}\hat \tau {\boldsymbol{\xi}}_1}{4|\boldsymbol{\xi}_1|^2 |\boldsymbol{\xi}_2|^2}.
\end{aligned}\end{equation*}
Then we have
\begin{equation*}\begin{aligned}
\sigma_4(\xi){\hat s}=&\left|\boldsymbol{\xi}_1\right|^2 {\hat s}_{\bar11}+\left|\boldsymbol{\xi}_2\right|^2{\hat s}_{\bar22} +2{\rm Re}\left( {\boldsymbol{\xi}}_2{\hat s}_{\bar21}\bar{\boldsymbol{\xi}}_1 \right) = \hat \tau.
\end{aligned}\end{equation*} If $\left|\boldsymbol{\xi}_1\right|=0,$ letting  $
{\hat s}_{\bar22}=\frac{\hat\tau}{\left|\boldsymbol{\xi}_2\right|^2},
$ and ${\hat s}_{\bar11},{\hat s}_{\bar21}$ be   arbitrarily chosen, we have  \begin{equation*}\begin{aligned}
\sigma_4(\xi){\hat s}=\left|\boldsymbol{\xi}_1\right|^2 {\hat s}_{\bar11}+\left|\boldsymbol{\xi}_2\right|^2{\hat s}_{\bar22} +2{\rm Re}\left( {\boldsymbol{\xi}}_2{\hat s}_{\bar21}\bar{\boldsymbol{\xi}}_1 \right) =\hat \tau.
\end{aligned}\end{equation*}Similarly, if $\left|\boldsymbol{\xi}_2\right|=0,$ letting  $
{\hat s}_{\bar11}=\frac{\hat \tau}{\left|\boldsymbol{\xi}_1\right|^2},
$ and ${\hat s}_{\bar22},{\hat s}_{\bar21}$ be  arbitrarily chosen,  we have $
\sigma_4(\xi){\hat s} 
=\hat \tau.
$
So $\sigma_4(\xi)$ is surjective. The theorem is proved.
\end{proof}
Now Theorem \ref{exact} follows from  Proposition \ref{complex} and Proposition  \ref{pell} directly.

\section{The boundary operator}
\subsection{The boundary complex}
To prove Hartogs--Bochner extension for  pluriharmonic  functions in two octonionic variables, it is necessary to find the first operator of the boundary complex of the octonionic Hessian   complex  (\ref{co}).

 Let us recall the definition of the boundary complex of a general differential complex on $\mathbb R^{N}$ (cf. \cite{AN,Naci}):
\begin{align*}
C^\infty\left(\mathbb R^{N},E^{(0)}\right)\xrightarrow{A_0(x,\partial)} C^\infty\left(\mathbb R^{N},E^{(1)}\right) \xrightarrow{A_1(x,\partial)}C^\infty\left(\mathbb R^{N},E^{(2)}\right) \xrightarrow{A_2(x,\partial)}  \cdots,
\end{align*}where $E^{(j)}$ are finite dimensional complex vector spaces, $A_j(x,\partial)$ are matrix-valued differential operators. 
Let $\Omega\subset \mathbb R^N$ be a domain with smooth boundary. Given an open set $U.$
We say $u\in C^\infty_c\left(\overline \Omega,E^{(j)}\right)$   has \emph{zero Cauchy data} on the boundary $\partial \Omega$ for $A_j(x,\partial)$ if for any $\psi\in C^\infty_c\left(U,E^{(j+1)}\right),$ i.e. $\psi$ is   compactly supported in a given open subset  $U,$ we have
\begin{align*}
\int_\Omega\left\langle A_j(x,\partial)u,\psi\right\rangle {\rm d}V=\int_\Omega\left\langle u,A_j^*(x,\partial)\psi\right\rangle {\rm d}V
\end{align*}
where $\langle\cdot,\cdot\rangle$ is the inner product of $E^{(j)},$ and $A_j^*(x,\partial)$ is the formal adjoint operator of $A_j(x,\partial).$
 Set
\begin{align*}
\mathcal J_{A_j}(\partial \Omega,U):=\left\{u\in C^\infty\left(U,E^{(j)}\right);u\ {\rm has\ zero\ Cauchy\ data\ on}\ \partial \Omega\ {\rm for}\ A_j(x,\partial)\right\}.
\end{align*}
  Then $U\rightarrow \mathcal J_{A_j}(\partial \Omega,U)$ is a sheaf. We must have
\begin{align*}
A_j(x,\partial)\mathcal J_{A_j}(\partial \Omega,U)\subset\mathcal J_{A_{j+1}}(\partial \Omega,U).
\end{align*}
This is because
\begin{equation*}\begin{aligned}
\int_\Omega\left\langle A_{j+1}(x,\partial)A_{j}(x,\partial)u,\psi\right\rangle{\rm d}V=&0=\int_\Omega\left\langle u,A^*_{j}(x,\partial)A^*_{j+1}(x,\partial)\psi\right\rangle{\rm d}V\\=&\int_\Omega\left\langle A_{j}(x,\partial)u,A^*_{j+1}(x,\partial)\psi\right\rangle{\rm d}V,
\end{aligned}\end{equation*}
for any $u\in\mathcal J_{A_j}(\partial \Omega,U)$ and compactly supported $\psi\in C^\infty\left(U,E^{(j+1)}\right),$ by $$A_{j+1}(x,\partial)A_j(x,\partial)=0,\quad A^*_j(x,\partial) A^*_{j+1}(x,\partial)=0.$$
Setting the quotient sheaf
\begin{align*}
Q^{(j)}(U)=\frac{\Gamma\left(U,E^{(j)}\right)}{\mathcal J_{A_j}\left(\partial \Omega,U\right)},
\end{align*}
we obtain a quotient complex of the form
\begin{align*}
Q^{(0)}(\partial \Omega)\xrightarrow{\widehat A_0} Q^{(1)}(\partial \Omega)\xrightarrow{\widehat A_1} Q^{(2)}(\partial \Omega)\xrightarrow{\widehat A_2}\dots
\end{align*}
where $\widehat A_j$ is the operator  induced by   $A_j(x, \partial),$ but is not necessarily a differential operator (cf. \cite{AN}).

 Denote $\partial_{10}:=\frac{\partial}{\partial x_1^0}. 
 $    Define  the vector fields $Z_{\bar j},Z_{j},\overleftarrow{Z_{\bar j}},\overleftarrow{Z_j}$ for $j=1,2,$ 
by 
\begin{equation}\begin{aligned}\label{Zj'}
 {Z_{\bar j}}f =&{\partial_{\bar j}}f-\rho_{\bar j}\cdot\left( \rho_{\bar1} ^{-1}\cdot {\partial_{\bar 1}}f\right), \qquad Z_jf=\partial_jf-\rho_j\cdot\left(\rho^{-1}_1\cdot\partial_1f \right), \\f\overleftarrow{Z_{\bar j}} =&f\overleftarrow{\partial_{\bar j}}-\left(f\overleftarrow{\partial_{\bar 1}}\cdot\rho_{\bar1}^{-1}\right)\cdot\rho_{\bar j}, \qquad f\overleftarrow{Z_{ j}} =f\overleftarrow{\partial_{ j}}-\left(f\overleftarrow{\partial_{ 1}}\cdot\rho_{ 1}^{-1}\right)\cdot\rho_{j},
\end{aligned}\end{equation}  
and define $T,\overline T,\overleftarrow{T},\overleftarrow{\overline T}$ by  \begin{equation}\begin{aligned}\label{T}{\overline T}f=&\rho_{\bar1}^{-1}\cdot {\partial_{\bar 1}}f-\partial_{10}f,\qquad  {T}f=\rho_{ 1}^{-1}\cdot {\partial_{  1}}f-\partial_{10}f,\\ f\overleftarrow{\overline T}=& f\overleftarrow{\partial_{\bar 1}}\cdot\rho_{\bar1}^{-1}-\partial_{10}f,\qquad f\overleftarrow{{T}}= f\overleftarrow{\partial_{  1}}\cdot\rho_{ 1}^{-1}-\partial_{10}f,\end{aligned}\end{equation}  
for   sufficiently smooth $\mathbb O$-valued or scalar  functions $f,$ where $$\rho_j=\partial_j\rho,\quad \rho_{\bar j}=\partial_{\bar j}\rho,\quad j=1,2.$$ They are all tangential operators  because $${Z_{\bar 2}}\rho=\rho\overleftarrow{Z_{\bar 2}}={Z_{  2}}\rho=\rho\overleftarrow{Z_{  2}}=T\rho=\overline{T}\rho=\rho\overleftarrow{T}=\rho \overleftarrow{\overline{T}}
=0,$$ by definition (\ref{Zj'}) and (\ref{defin'}) of $\rho$. Obviously only for $j=2$ operators do not vanish, since   $$Z_{\bar 1}=Z_1=\overleftarrow{Z_{\bar 1}}=\overleftarrow{Z_1}=0,$$ in (\ref{Zj'}) by alternativity.   
Then, left and right octonionic Dirac operators can be written as \begin{equation}\begin{aligned}\label{nablaZ'} 
 {\partial_{\bar j}}f=& {Z_{\bar j}}f +\rho_{\bar j} \cdot {\overline T}f+\rho_{\bar j}\cdot\partial_{10}f ,\quad {\partial_{  j}}f= {Z_{  j}}f +\rho_{  j} \cdot {  T}f+\rho_{  j}\cdot\partial_{10}f
, 
\\ f\overleftarrow{\partial_{\bar j}}=& f\overleftarrow{Z_{\bar j}} +f\overleftarrow{\overline T}\cdot\rho_{\bar j} +\partial_{10}f\cdot\rho_{\bar j}, \quad f\overleftarrow{\partial_{  j}}= f\overleftarrow{Z_{  j}} +f\overleftarrow{  T}\cdot\rho_{  j} +\partial_{10}f\cdot\rho_{  j}. 
\end{aligned}\end{equation}

\subsection{The characterization of the sheaves $\mathcal J_j(U)$ 
}
\begin{prop}\label{J13}Suppose that $U\cap\partial \Omega\neq\emptyset$. Then,
\\
{\rm (1)}
$\mathcal J_0(U)$ consists of real functions $f$ such that $f=O\left(\rho^2\right) $; \\
{\rm (2)} $\mathcal J_1(U)$ consists of  \begin{equation}\begin{aligned}\label{J3'}
u=\left(\begin{matrix} \left|\rho_{1}\right|^2    &\rho_{\bar 1}  \rho_{2} \\\rho_{\bar 2}\rho_{  1}   & \left|\rho_{2} \right|^2   \end{matrix}\right)\cdot H \quad \left({\rm mod} \ \rho\mathcal H_2(\mathbb O)\right), 
\end{aligned}\end{equation}
 for some  $H\in C^\infty\left(U,\mathbb R\right).$
\end{prop}
\begin{proof} (1) By definition,   $\mathcal J_0(U)$ is the set of  all $f\in C^\infty\left(U,\mathbb R\right)$ satisfying
\begin{align}\label{D0bp}
 \int_\Omega\left\langle \mathcal D_0f ,u\right\rangle {\rm d}V=\int_\Omega\left\langle f,\mathcal D_0^*u \right\rangle {\rm d}V,
\end{align}for any $u\in C_c\left(U,\mathcal H_2(\mathbb O)\right)$. By definition   (\ref{inner}), we have
\begin{equation}\begin{aligned}\label{u}
\left(\mathcal D_0f,u\right) 
=\int_{\Omega}\left(\left\langle\Delta_1f , u_{\bar 1 1} \right\rangle+\left\langle{\partial_{ \bar 1}}f\overleftarrow{\partial_{ 2}}, u_{\bar 1 2} \right\rangle+\left\langle{\partial_{ \bar 2}}f\overleftarrow{\partial_{ 1}}, u_{\bar 21} \right\rangle+\left\langle\Delta_2f , u_{\bar 2 2} \right\rangle\right){\rm d}V,\end{aligned}\end{equation} 
where 
\begin{equation*}\begin{aligned}\label{dfdu} \int_\Omega\left\langle{\partial_{ \bar i}}f\overleftarrow{\partial_{j}}, u_{\bar ij} \right\rangle {\rm d}V=&-\int_\Omega\left\langle f\overleftarrow{\partial_{j}}, {\partial_{i}}u_{\bar ij} \right\rangle {\rm d}V+\int_{\partial\Omega}\left\langle \rho_{\bar i} \cdot f\overleftarrow{\partial_{ j}}, u_{\bar ij} \right\rangle  \frac{{\rm d}S}{|{\rm grad}\rho|}\\=  &\int_\Omega\left\langle f, \left({\partial_{i}}u_{\bar ij} \right) \overleftarrow{\partial_{\bar j}}\right\rangle {\rm d}V+\int_{\partial\Omega}\left( \left\langle\rho_{\bar i} \cdot \partial_j f, u_{\bar i j} \right\rangle -\left\langle f\cdot\rho_j, {\partial_i}u_{\bar ij} \right\rangle \right)\frac{{\rm d}S}{|{\rm grad}\rho|},  
\end{aligned}\end{equation*}
by  using  integration by parts  (\ref{bp}) twice. Here $ f \overleftarrow{\partial_{ j}}=\partial_j f$ since $f$ is real. So we get  
 \begin{equation}\begin{aligned}\label{D0bp1}\left(\mathcal D_0f,u\right)
=  &\left( f,\mathcal D_0^* u\right)-\sum_{i,j=1}^2\int_{\partial\Omega}\left(\left\langle f \rho_j, {\partial_i}u_{\bar ij} \right\rangle+\left\langle 
\rho_{\bar i} \cdot\partial_j f, u_{\bar ij} \right\rangle\right)\frac{{\rm d}S}{|{\rm grad}\rho|}, 
\end{aligned}\end{equation}by expression of $\mathcal D_0^*$ in Proposition \ref{Dstar}. It is easy to see that  if $f=O(\rho^2), $ then (\ref{D0bp1}) implies (\ref{D0bp}). On the other hand,   if  we take $u_{\bar ij}=\rho\cdot v_{\bar ij},$ with $\left(v_{\bar ij}\right)\in C_c^2\left(U,\mathcal H_2(\mathbb O)\right)$ in (\ref{D0bp1}), then (\ref{D0bp}) implies
\begin{align*}
\sum_{i,j=1}^2\int_{\partial\Omega}\left\langle f\rho_j,\rho_iv_{\bar ij}\right\rangle\frac{{\rm d}S}{|{\rm grad}\rho|} 
=0.
\end{align*} Taking  $v_{\bar ij}= f\rho_{\bar i}\rho_j,$ we get \begin{align*}\int_{\partial\Omega} |f|^2|{\rm grad}\rho|^4 \frac{{\rm d}S}{|{\rm grad}\rho|}=0,\end{align*}
thus we have $f\equiv0$ on $\partial\Omega.$ Now (\ref{D0bp}) together with (\ref{D0bp1}) implies   \begin{align*}
\int_{\partial\Omega}{\rm tr}\left(\left(\rho_{\bar i}\cdot\partial_jf\right)\left(u_{\bar ij}\right)\right)\frac{{\rm d}S}{|{\rm grad}\rho|}=0,
\end{align*}
for any $\left(u_{\bar ij}\right)\in C_c^2\left(U,\mathcal H_2(\mathbb O)\right).$ Hence\begin{align*}
\rho_{\bar i}\cdot\partial_jf=0
\end{align*}on $\partial \Omega$ for any $i,j=1,2.$ 
Taking $i$ such that $\left|\rho_{i}\right| \neq0$ in a neighbourhood, we get   
$\left.\left|\rho_{i}  \right|^2 \cdot \partial_jf \right|_{\partial \Omega}=0,$    thus   we have    
 $\left.\partial_{ j}f \right|_{\partial \Omega\cap U}=0.$ So $f=O\left(\rho^2\right).$ \\
(2) By definition, $\mathcal J_1(U)$ is the set of all  $u\in C^\infty\left(U,\mathcal H_2(\mathbb O)\right)$ satisfying
\begin{align*}
 \int_\Omega\left\langle \mathcal D_1u,v \right\rangle {\rm d}V= \int_\Omega\left(\left\langle  u_{\bar 1 2}\overleftarrow{\partial_{ \bar 2}}-{\partial_{\bar 1}}u_{\bar 2 2},v_1 \right\rangle+\left\langle u_{ \bar 2 1}\overleftarrow{\partial_{\bar 1}}-{\partial_{\bar 2}}u_{\bar 11},v_2 \right\rangle\right) {\rm d}V =\int_\Omega\left\langle u,\mathcal D_1^*v \right\rangle {\rm d}V,
\end{align*}for any $v=\left(\begin{matrix}v_1\\v_2\end{matrix}\right)\in C_c^2\left(U, \mathbb O^2\right)$. 
For  $v_1,v_2 \in C_c^2\left(U,\mathbb O\right)$ with compact support, we have
\begin{equation*}\begin{aligned}
\int_\Omega\left\langle u_{\bar ij}\overleftarrow{\partial_{\bar j}},v_i\right\rangle{\rm d}V=&-\int_\Omega\left\langle u_{\bar ij},v_i\overleftarrow{\partial_{j}}\right\rangle{\rm d}V+\int_{\partial \Omega}\left\langle u_{\bar ij}\cdot\rho_{\bar j} ,v_i \right\rangle\frac{{\rm d}S}{|{\rm grad}\rho|}, 
\\ \int_\Omega\left\langle {\partial_{\bar i}} u_{\bar jj},v_i\right\rangle{\rm d}V=&-\int_\Omega\left\langle u_{\bar jj},{\partial_{i}}  v_i\right\rangle{\rm d}V+\int_{\partial \Omega}\left\langle \rho_{\bar i} \cdot u_{\bar jj},v_i \right\rangle\frac{{\rm d}S}{|{\rm grad}\rho|}\\=&-\int_\Omega\left\langle u_{\bar jj},\frac{{\partial_{i}}  v_i+\bar v_i\overleftarrow{\partial_{\bar i}}}{2}\right\rangle{\rm d}V+\int_{\partial \Omega}\left\langle \rho_{\bar i} \cdot u_{\bar jj},v_i\right\rangle\frac{{\rm d}S}{|{\rm grad}\rho|},  
\end{aligned}\end{equation*} for $i\neq j,$
by  using   integration by parts  (\ref{bp}). Then, we get 
\begin{equation*}\begin{aligned}
\left(\mathcal D_1u,v\right)=\int_\Omega\left\langle u,\mathcal D_1^*v\right\rangle  {\rm d}V&+ \int_{\partial \Omega}\left(\left\langle u_{\bar 12}\cdot\rho_{\bar 2} -\rho_{\bar 1} \cdot u_{\bar 22},v_1 \right\rangle + \left\langle u_{\bar 21}\cdot\rho_{\bar 1} -\rho_{\bar 2} \cdot u_{\bar 11},v_2 \right\rangle\right)\frac{{\rm d}S}{|{\rm grad}\rho|},
\end{aligned}\end{equation*}  by   expression    $\mathcal D_1^*$   in (\ref{ddd}) and (\ref{2.7}).
  Since $v_{1}$ and $v_{2}$  are arbitrarily  chosen, $u\in\mathcal J_1(U)$  if and only if
\begin{equation}\begin{aligned}\label{j1}
 u_{\bar 12}\cdot\rho_{\bar 2} -\rho_{\bar 1} \cdot u_{\bar 22} =0,\quad  u_{\bar 21}\cdot\rho_{\bar 1} -\rho_{\bar 2} \cdot u_{\bar 11} =0,\quad {\rm on}\ \partial \Omega.
\end{aligned}\end{equation} 
It is obvious that $u$ satisfying (\ref{J3'}) must satisfy  (\ref{j1}). Now suppose $u$ satisfies (\ref{j1}). Then, 
\begin{equation}\begin{aligned}\label{122}
\left.\rho_{  1} \cdot\left(u_{\bar 12}\cdot\rho_{\bar 2} \right)\right|_{\partial \Omega}=&\left.\left|\rho_{ 1} \right|^2 \cdot u_{\bar 22}\right|_{\partial \Omega} ,\\ \left.\rho_{2} \cdot \left(u_{\bar 21}\cdot\rho_{\bar 1} \right)\right|_{\partial \Omega}=&\left.\left|\rho_{2} \right|^2 \cdot u_{\bar 11}\right|_{\partial \Omega},
\end{aligned}\end{equation}which are both real, by alternativity.
This together with  $
{\rm Re}\left(\rho_{ 1} \cdot\left(u_{\bar 12}\cdot\rho_{\bar 2} \right)\right)={\rm Re}\left(\rho_{2} \cdot \left(u_{\bar 21}\cdot\rho_{\bar 1} \right)\right) 
$ by Lemma \ref{asso} (iii), implies  
\begin{align}
\left.\left|\rho_{ 1} \right|^2\cdot u_{\bar 22}\right|_{\partial \Omega}=\left.\left|\rho_{ 2} \right|^2 \cdot u_{\bar 11}\right|_{\partial \Omega}.
\end{align}  Hence   for $\left|\rho_{  1}\right|\neq0,\left| \rho_{2} \right|\neq0,$ we have \begin{align} 
\left.\frac{u_{\bar 22}}{\left|\rho_{ 2} \right|^2} \right|_{\partial \Omega}=\left.\frac{u_{\bar 11}}{\left|\rho_{  1} \right|^2} \right|_{\partial \Omega}=:H.
\end{align}
So we must have \begin{align}\label{uo}
  u_{\bar ij}=\rho_{\bar i}\rho_{j} H+O(\rho),
  \end{align} by (\ref{j1}). Thus,      (\ref{J3'}) follows. If $\left|\rho_{  1} \right|=0,$ we must  have $\left|\rho_{2} \right|\neq0,$ then $\left.u_{\bar 11}\right|_{\partial\Omega}=\left.u_{\bar 12}\right|_{\partial\Omega}=0$ by (\ref{j1}). Then,  (\ref{J3'}) still holds by letting $H=\frac{u_{\bar22}}{\left|\rho_{ 2} \right|^2}.$ The case of  $\left|\rho_{2} \right|=0$ is similar.
The  proposition is proved. \end{proof}

\subsection{The tangential octonionic Hessian operator}
\begin{lem}For $\mathbb R$-valued function $u_0,u_1 $  independent of $x_{1}^0,$ we have   
 \begin{equation*}\begin{aligned}
\bigg[2\partial_{\bar i}\partial_j\left(u_0+\rho u_1\right)\bigg]_{2\times2}   = &\bigg[\left(Z_{\bar i}u_0  + \rho_{\bar i}\cdot\overline Tu_0\right) \overleftarrow{Z_{j}}
+Z_{\bar i}\left(u_0\overleftarrow{Z_j}+ u_0\overleftarrow{  T}\cdot\rho_j\right)\\
&+u_1\left(\rho_{\bar i}\overleftarrow{Z_j}+Z_{\bar i}\rho_j\right)+\Phi_{\bar i}\rho_j+\rho_{\bar i}\Phi_{ j}-\rho_{\bar i}\rho_j\Psi \bigg]_{2\times 2}
  \quad \left({\rm mod}\ \big[\rho_{\bar i}\rho_j\big]_{2\times 2},\ \rho\right),
\end{aligned}\end{equation*} where    \begin{equation}\begin{aligned}\label{Phij}
\Phi_{  j}:=&  \overline T\left(u_0 \overleftarrow{Z_{  j}}\right) +\overline T \left(u_0  \overleftarrow{T}\cdot \rho_{  j}  \right)  + 2 u_1 \overleftarrow{Z_{j}}  +u_1 \overleftarrow{T}\cdot\rho_j+u_1\cdot \overline T\rho_j ,\\\Psi:=& \frac{2{\rm Re}\left(\rho_{\bar 1}\Phi_{1}\right)}{\left|\rho_1\right|^2},
\end{aligned}\end{equation} and   $\Phi_{\bar j}:=\overline{\Phi_{  j}}.$   
\end{lem}
\begin{proof}
For $\mathbb R$-valued function  $u_0,$ which is  independent of $x_{1}^0,$  
 we have \begin{equation}\begin{aligned}\label{pu0}
 2\partial_{\bar i}\partial_ju_0=&\left(\partial_{\bar i}u_0\right)\overleftarrow{\partial_j}+\partial_{\bar i}\left(u_0\overleftarrow{\partial_j}\right) \\=&\left(\partial_{\bar i}u_0\right)\overleftarrow{Z_j}+\left(\left(\partial_{\bar i}u_0\right)\overleftarrow{  T}\right)\rho_j +Z_{\bar i}\left(u_0\overleftarrow{\partial_{j}}\right)+\rho_{\bar i}\left(\overline T\left(u_0\overleftarrow{\partial_{j}}\right)\right) \\=&\left(Z_{\bar i}u_0\right)\overleftarrow{Z_j} +\left(\rho_{\bar i}\cdot\overline Tu_0 \right)\overleftarrow{Z_{j}}+\left[\left(Z_{\bar i}u_0\right)\overleftarrow{  T} +\left(\rho_{\bar i}\cdot\overline Tu_0 \right)\overleftarrow{  T} \right]\rho_j\\&+Z_{\bar i}\left(u_0\overleftarrow{Z_j}\right)+Z_{\bar i}\left( u_0\overleftarrow{  T}\cdot\rho_j\right)+\rho_{\bar i}\left[\overline T\left(u_0\overleftarrow{Z_j}\right)+  \overline T\left( u_0\overleftarrow{  T}\cdot\rho_j\right)\right], 
\end{aligned}\end{equation}by (\ref{nablaZ'}), where the second and third identity  hold by using $$\partial_{10}\left(\partial_{\bar i}u_0\right)=\partial_{\bar i}\left(\partial_{10}u_0\right) =0.$$
For $\mathbb R$-valued function $u_1$ independent of $x_1^0,$ we have 
  \begin{equation}\begin{aligned}\label{pu1}
 2\partial_{\bar i}\partial_j\left(\rho u_1\right)=&\left(\partial_{\bar i}\left(\rho u_1\right)\right)\overleftarrow{\partial_j}+\partial_{\bar i}\left(\left(\rho u_1\right)\overleftarrow{\partial_j}\right)\\=&\left(\rho_{\bar i}u_1+\rho\cdot\partial_{\bar i}u_1\right)\overleftarrow{\partial_j}+\partial_{\bar i}\left(u_1\rho_j+\rho\cdot u_1\overleftarrow{\partial_j}\right)\\=& 2\rho_{\bar i}\cdot u_1\overleftarrow{\partial_j}+2\partial_{\bar i}u_1\cdot\rho_j+u_1\left(\rho_{\bar i}\overleftarrow{\partial_{j}}+ \partial_{\bar i}\rho_j\right) 
  \qquad  (   {\rm mod}\ \rho)\\=&2\rho_{\bar i}\left(u_1 \overleftarrow{Z_j}+u_1 \overleftarrow{T}\cdot\rho_j\right)+2\left(Z_{\bar i}u_1+\rho_{\bar i}\cdot\overline Tu_1\right)\rho_j \\& + u_1\left(\rho_{\bar i}\overleftarrow{Z_j}+\rho_{\bar i}\overleftarrow{  T}\cdot \rho_j+ Z_{\bar i} \rho_j +    \rho_{\bar i}\cdot \overline T\rho_j  \right) \qquad   ( {\rm mod}\ \rho), 
\end{aligned}\end{equation} 
where the third identity hold by $$\partial_{10}\left(\rho_{\bar i}u_1\right)=\partial_{10}\left(\partial_{\bar i}\rho\right)u_1+\rho_{\bar i}\partial_{10}u_1=\partial_{\bar i}\left(\partial_{10}\rho\right)u_1=0.$$  
Now the summation of (\ref{pu0}) and (\ref{pu1}) gives us  \begin{equation*}\begin{aligned}\label{cau}
 \bigg[2\partial_{\bar i}\partial_j\left(u_0+\rho u_1\right)\bigg]_{2\times2} =&\bigg[\left(Z_{\bar i}u_0\right)\overleftarrow{Z_j} +\left(\rho_{\bar i}\cdot\overline Tu_0\right) \overleftarrow{Z_{j}}
+Z_{\bar i}\left(u_0\overleftarrow{Z_j}\right)+Z_{\bar i}\left( u_0\overleftarrow{  T}\cdot\rho_j\right)\\
&\qquad +u_1\left(\rho_{\bar i}\overleftarrow{Z_j}+Z_{\bar i}\rho_j\right)+\Phi_{\bar i}\rho_j+\rho_{\bar i}\Phi_{ j}\bigg]_{2\times 2}  \\=&\bigg[\left(Z_{\bar i}u_0\right)\overleftarrow{Z_j} +\left(\rho_{\bar i}\cdot\overline Tu_0\right) \overleftarrow{Z_{j}}
+Z_{\bar i}\left(u_0\overleftarrow{Z_j}\right)+Z_{\bar i}\left( u_0\overleftarrow{  T}\cdot\rho_j\right)\\
& +u_1\left(\rho_{\bar i}\overleftarrow{Z_j}+Z_{\bar i}\rho_j\right)+\Phi_{\bar i}\rho_j+\rho_{\bar i}\Phi_{ j}-\rho_{\bar i}\rho_j\Psi\bigg]_{2\times 2} 
,  \quad \left({\rm mod}\ \big[\rho_{\bar i}\rho_j\big]_{2\times 2}, \rho\right)
\end{aligned}\end{equation*} where   
$\Phi_j$ and $\Psi$ are  given in (\ref{Phij}). We add the term  $\rho_{\bar i}\rho_j\Psi$ in the last identity of (\ref{cau})   in order to make the $(1,1)$-th entry of $\big[2\partial_{\bar i}\partial_j\left(u_0+\rho u_1\right)\big]_{2\times2}$ vanishes after  ${\rm mod}\ \big[\rho_{\bar i}\rho_j\big]_{2\times 2},\   \rho$, since 
$$\Phi_{\bar 1}\rho_1+\rho_{\bar 1}\Phi_{ 1}-\rho_{\bar 1}\rho_1\Psi=\Phi_{\bar 1}\rho_1+\rho_{\bar 1}\Phi_{ 1}-\rho_{\bar 1}\rho_1\frac{2{\rm Re}\left(\rho_{\bar 1}\Phi_{1}\right)}{\left|\rho_1\right|^2}=0.$$
The proposition is proved. \end{proof} 

Noting that ${Z}_1=0=\overleftarrow{Z_1}$ by definition (\ref{Zj'}), 
   ${\rm Re} (\left(Z_{\bar 2}u_0 )\overleftarrow{Z_2}\right)= {\rm Re}\left(Z_{\bar 2}\left({Z_2}u_0\right)\right),$ and \begin{equation}\begin{aligned}
\Phi_{\bar 1}\rho_2+\rho_{\bar 1}\Phi_{2}-\rho_{\bar 1}\rho_2\Psi=&\Phi_{\bar 1}\rho_2+\rho_{\bar 1}\Phi_{2}-\frac{2{\rm Re}\left(\rho_{\bar 1}\Phi_{1}\right)}{\left|\rho_1\right|^2}\rho_{\bar 1}\rho_2=\rho_{\bar 1}\Phi_2-\frac{\left(\rho_{\bar 1}\Phi_1\rho_{\bar 1}\right)\rho_2}{\left|\rho_1\right|^2},\\\Phi_{\bar 2}\rho_2+\rho_{\bar 2}\Phi_{2}-\rho_{\bar 2}\rho_2\Psi=&2{\rm Re}\left(\Phi_{\bar 2}\rho_2\right) -\frac{2{\rm Re}\left(\rho_{\bar 1}\Phi_{1}\right)\left|\rho_2\right|^2}{\left|\rho_1\right|^2}  ,
\end{aligned}\end{equation} we have \begin{equation*}\begin{aligned}
\left(\begin{matrix} {\partial_{\bar 1}} {\partial_{1}}\left(u_0+\rho u_1\right) & 
{\partial_{\bar 1}} {\partial_{2}} \left(u_0+\rho u_1\right) \\ {\partial_{\bar 2}} {\partial_{1}}\left(u_0+\rho u_1\right) & {\partial_{\bar 2}} {\partial_{2}} \left(u_0+\rho u_1\right) \end{matrix}\right) =\frac12\left(F_{\bar ij}\right)  
\quad\left ({\rm mod}\ \mathcal J_1(U)\right),
\end{aligned}\end{equation*} where \begin{equation}\begin{aligned}\label{fij'}
F_{\bar 11}=&0,\\
F_{\bar 12}=&   \left(\rho_{\bar 1}\cdot\overline{T}u_0\right) \overleftarrow{Z_{2}}
  +u_1\cdot \rho_{\bar 1}\overleftarrow{Z_2} +\rho_{\bar 1}\Phi_{  2}-\frac{\left(\rho_{\bar 1}\Phi_1\rho_{\bar 1}\right)\rho_2}{\left|\rho_1\right|^2},\\F_{\bar 21}=&\overline{F_{\bar 12}} 
   ,\\ F_{\bar 22}=& 2 {\rm Re}\left[ 
   Z_{\bar 2}\left( {Z_2}u_0\right)+\left(\rho_{\bar 2}\left(\overline{T}u_0\right)\right)\overleftarrow{Z_{2}}
  +u_1  Z_{\bar 2}\rho_2+\Phi_{\bar 2}\rho_2-\frac{ \left|\rho_2\right|^2\rho_1\Phi_{\bar 1} }{\left|\rho_1\right|^2}\right],
\end{aligned}\end{equation} for  \begin{equation*}\begin{aligned}\Phi_{  1}=&  \overline T \left(u_0  \overleftarrow{T}\cdot \rho_{  1}  \right)    +u_1 \overleftarrow{T}\cdot\rho_1+u_1\cdot \overline T\rho_1 ,\\\Phi_{  2}=&  \overline T\left(u_0 \overleftarrow{Z_{  2}}\right) +\overline T \left(u_0  \overleftarrow{T}\cdot \rho_{2}  \right)  + 2 u_1 \overleftarrow{Z_{2}}  +u_1 \overleftarrow{T}\cdot\rho_2+u_1\cdot \overline T\rho_2 .\end{aligned}\end{equation*}    

 By the characterization of $\mathcal J_0(U)$ and $\mathcal J_1(U)$ in Proposition \ref{J13}, we have the isomorphism  \begin{equation*}\begin{aligned}
\pi_0:\Gamma\left(U,\mathbb R\right)/\mathcal J_0(U)&\xrightarrow{\cong} \Gamma\left(U\cap\partial\Omega,\mathbb R\oplus\mathbb R \right),\\\hat f=  u_0+ u_1+O\left(\rho^2\right)&\mapsto\left(  u_0, u_1\right),
\end{aligned}\end{equation*} 
and\begin{equation*}\begin{aligned}\pi_1:\Gamma\left(U,\mathcal H_2(\mathbb O)\right)/\mathcal J_1(U)&\xrightarrow{\cong} \Gamma\left(U\cap\partial\Omega,\mathcal H'_2(\mathbb O) %\oplus\mathcal H_2(\mathbb O)
 \right),\\\widehat F_{\bar ij}&\mapsto \frac12F_{\bar ij} 
 , \end{aligned}\end{equation*}   given by  
 (\ref{fij'}), with $\mathcal H'_2(\mathbb O)$ given by (\ref{h2o'}).
We call the operator  induced from octonionic Hessian  operators the \emph{tangential octonionic Hessian operators}: \begin{equation*}\begin{aligned} {\mathscr D}_0:\Gamma\left(U\cap\partial \Omega,\mathbb R^2\right) &\rightarrow \Gamma\left(U\cap\partial\Omega,\mathcal H'_2(\mathbb O)%\oplus\mathcal H_2(\mathbb O)
\right),\\\hat f=\left(u_0, u_1\right)&\mapsto\pi_1\left(\mathcal D_0\hat f\right)=%\left(F^1_{\bar ij},F^2_{\bar ij} \right),
\frac12F_{\bar ij},\end{aligned}\end{equation*} where $F_{\bar ij}$ are given by (\ref{fij'}).   
Then  the   system of equation $\mathscr D_0\left(\begin{matrix}u_0 \\ u_1\end{matrix}\right)=0$  reduces to the system of equations $F_{\bar12}=0,F_{\bar22}=0,$
 which is the octonionic Hessian version of the following system (\ref{na}) given by Andreotti--Nacinovich \cite{AN}.

\begin{rem}   The system of equation $\left(\partial\bar\partial\right)_b\left(u_0+\rho u_1\right)=0$ can be  reduces to the system of equations \begin{equation}\begin{aligned}\label{na}%\left\{
\mathscr L_{ij}u_0+l_{ij}u_1=&0,\quad &&1\leq i,j\leq n-1,\\
S_iu_0+T_iu_1=&0,\quad &&1\leq i\leq n-1, 
\end{aligned} \end{equation}where $\rho=y_n-\phi\left(z_1,\dots,z_{n-1},x_n\right),$ \begin{equation*}\begin{aligned}\mathscr L_{ij}:=&\frac{\partial^2}{\partial z_i\partial \bar z_j}+\frac12\bar\alpha_j\frac{\partial^2}{\partial z_i\partial x_n}+\frac{\partial^2}{\bar\partial z_i\partial x_n}+\frac14\alpha_i\bar\alpha_j\frac{\partial^2}{\partial x^2_n},\\ S_i:=&a\left[\frac{\partial^2}{\bar\partial z_i\partial x_n}+\frac14\alpha_i\bar\alpha_j\frac{\partial^2}{\partial x^2_n}\right],\\T_i:=&\frac{\partial}{\partial\bar z_i}+\frac12\bar \alpha_i \frac{\partial}{\partial x_n},\end{aligned}\end{equation*}  with  $a=-2\left(\mathbf i+\frac{\partial\phi}{\partial x_n}\right)^{-1},\alpha_j=a\frac{\partial\phi}{\partial z_j},l_{ij}=\frac{\partial^2\rho}{\partial z_i\partial{\bar z_j}}.$
 \end{rem}

   \begin{prop}\label{phat}
If $\hat f\in C^\infty\left(\partial \Omega,\mathbb R^2\right)$ is tangentially OPH, i.e. $  {\mathscr D}_0\hat f=0,$ there exists a representative $\tilde  f\in\Gamma\left(\overline \Omega,\mathbb R\right)$ such that $\left.\tilde f\right|_{\partial \Omega}=\hat f$ and $\mathcal D_0\tilde f$ is flat on $\partial \Omega.$
\end{prop}
\begin{proof}
Since $\hat f$ is tangentially OPH on $\partial \Omega,$ $\mathscr D_0\hat f =0.$
 Let $\tilde f=  u_0 +\rho u_1$ be
a  smooth extension of $\hat f$ to a neighborhood of $\overline \Omega$, we see that %$\pi_1\left(\mathcal D_0\tilde f\right)=0$ by (\ref{f12}), i.e. 
 $\mathcal D_0\tilde f   
 \in\mathcal J_1(U).$  Thus  
\begin{equation*}\begin{aligned}
\mathcal D_0\tilde f=\left(\begin{matrix}\Delta_1\left(u_0+\rho u_1\right) &{\partial_{\bar1}} {\partial_2} \left(u_0+\rho u_1\right) \\{\partial_{\bar2}} {\partial_1}\left(u_0+\rho u_1\right) &\Delta_2\left(u_0+\rho u_1\right) \end{matrix}\right) =\left(\begin{matrix} \left|\rho_{  1}\right|^2  &\rho_{\bar 1} \rho_{2} \\\rho_{\bar 2}  \rho_{  1} & \left|\rho_{2}\right|^2   \end{matrix}\right)H+\rho\cdot G, 
\end{aligned}\end{equation*}
for   some $H\in\Gamma(U,\mathbb R)$ and $G\in\Gamma\left(U,\mathcal H_2(\mathbb O)\right), $   by the characterization of $\mathcal J_1(U)$ in (\ref{J3'}).
Then if set $\tilde u :=u_0+\rho u_1-\frac12\rho^2H,$ we get \begin{equation*}\begin{aligned}
\left(\begin{matrix}\Delta_1\tilde u  &{\partial_{\bar1}} {\partial_2}\tilde u \\{\partial_{\bar2}} {\partial_1}\tilde u &\Delta_2\tilde u  \end{matrix} \right) = \rho\cdot G.
\end{aligned}\end{equation*}
As $\mathcal D_1\circ\mathcal D_0\tilde u  =0$ automatically, we have 
\begin{equation*}\begin{aligned}
 \left.\left[G_{\bar 12}\cdot\rho_{\bar 2}-\rho_{\bar 1}\cdot G_{\bar 22}\right]\right|_{\partial \Omega}=&0,\\ \left.\left[G_{\bar 21}\cdot\rho_{\bar 1}-\rho_{\bar 2}\cdot G_{\bar 11}\right]\right|_{\partial \Omega}=&0. 
\end{aligned}\end{equation*}
Then  
\begin{equation*}\begin{aligned}
G=\left(\begin{matrix} \left|\rho_{ 1}\right|^2    &\rho_{\bar 1} \rho_{2} \\\rho_{\bar 2} \rho_{  1} & \left|\rho_{2} \right|^2  \end{matrix}\right)H_1+\rho\cdot G',
\end{aligned}\end{equation*}
by the same procedure as  (\ref{j1})-(\ref{uo}), for   some $H_1\in\Gamma(U,\mathbb R),G'\in\Gamma\left(U,\mathcal H_2(\mathbb O)\right).$ Thus if write $\tilde{\tilde u}=u_0+\rho u_1-\frac12\rho^2H-\frac16\rho^3H_1,$ we have 
\begin{equation*}\begin{aligned}
\left(\begin{matrix}\Delta_1\tilde{\tilde u}  &{\partial_{\bar1}} {\partial_2}\tilde{\tilde u} \\{\partial_{\bar2}} {\partial_1}\tilde{\tilde u} &\Delta_2\tilde{\tilde u}  \end{matrix}\right) = \rho\cdot G'.
\end{aligned}\end{equation*}
 Repeating this procedure, we get
\begin{align*}
\left(\begin{matrix}\Delta_1\left(u_0+\rho u_1+\dots\right) &{\partial_{\bar1}} {\partial_2}\left(u_0+\rho u_1+\dots\right) \\{\partial_{\bar2}} {\partial_1}\left(u_0+\rho u_1+\dots\right) &\Delta_2\left(u_0+\rho u_1+\dots\right)  \end{matrix}\right)=O_{\partial \Omega}^\infty.
\end{align*}with a formal power series in $\rho$ with coefficients $C^\infty$ on $\partial\Omega$, where $O_{\partial \Omega}^\infty$ denotes functions vanishing of infinite order on $\Omega$. By using Whitney extension theorem as in \cite{AH1,AH2}, we get the conclusion.  The proposition is proved. 
 \end{proof}

\begin{rem}
 Such extension for CR functions was constructed by Andreotti--Hill \cite{AH1,AH2},  while for pluriharmonic functions satisfying $\partial\bar\partial$-equation was constructed by Andreotti--Nacinovich in \cite{AN}.  This generalized to     $k$-CF functions   by the second named  author in \cite{wang29}.  Recently, such extension for tangentially monogenic functions was constructed by us in \cite{SW3}.
\end{rem} 
 \section{Non-homogeneous octonionic Hessian  equations and extension theorem}
\subsection{Solutions of non-homogeneous octonionic Hessian  equations}
The natural Hodge-Laplacian associated to the differential  complex (\ref{co}) should be
\begin{align*}
\widetilde\Box_j:=\mathcal D_{j-1}\mathcal D_{j-1}^*+\mathcal D_j^*\mathcal D_j.
\end{align*}
 But   $\mathcal D_0,\mathcal D_2$ and ${\mathcal D_4}$ are  differential operators of second order, while $\mathcal D_1,{\mathcal D}_3$ are of first order. Thus  $\widetilde\Box_j$   have degenerate principal symbols and  are not uniformly elliptic. So it is better to  consider the associated Hodge-Laplacians of fourth order as \begin{equation}\begin{aligned}\label{hodge}
\Box_0:=& \mathcal D_0^*\mathcal D_0 ,\\\Box_1:=&\mathcal D_0\mathcal D_0^*+\left(\mathcal D_1^*\mathcal D_1\right)^2,\\\Box_2:=&\left(\mathcal D_1\mathcal D_1^*\right)^2 +{\mathcal{D}_2^*}{\mathcal{D}_2},\\\Box_3:=&\mathcal D_2\mathcal D_2^* +\left({\mathcal{D}_3^*}{\mathcal{D}_3}\right)^2 ,\\\Box_4:=&\left(\mathcal D_3\mathcal D_3^*\right)^2 +{\mathcal{D}_4^*}{\mathcal{D}_4}.
\end{aligned}\end{equation}  It is easy to see that  $\Box_0=\Delta^2$ by Corollary \ref{p23}.
\begin{prop}\label{boxj}
Let $L_j(   {\xi}):=\sigma\left(\Box_j\right)(   {\xi})$ and $r_j={\rm dim} \mathcal V_j .$ Then
\begin{equation*}\begin{aligned}
L_0(   {\xi})=& \sigma_0(   {\xi})^*  \sigma_0(   {\xi}) ,\\L_1(   {\xi})=&\sigma_0(   {\xi})  \sigma_0(  {\xi})^*+\left(\sigma_1(  {\xi})^*  \sigma_1(   {\xi})\right)^2,\\L_2(  {\xi})=&\left(\sigma_1(   {\xi})  \sigma_1(  {\xi})^*\right)^2+\sigma_2(   {\xi})^*  \sigma_2(  {\xi}) ,\\L_3(  {\xi})=&\sigma_2(   {\xi})  \sigma_2(  {\xi})^*+\left(\sigma_3(   {\xi})^*  \sigma_3(  {\xi})\right)^2,\\L_4(   {\xi})=&\left(\sigma_3(  {\xi})  \sigma_3(   {\xi})^*\right)^2+\sigma_4(  {\xi})^*  \sigma_4(  {\xi}) ,
\end{aligned}\end{equation*}
are all automorphisms of $\mathbb C^{r_j}$,  respectively, for any fixed dual variable $   {\xi}\in\mathbb O^{2}\setminus\{\mathbf 0\},$   and are all homogeneous of degree $4$ in $   {\xi}.$
\end{prop}
The proof is similar to \cite{Wa10} for $k$-Cauchy--Fueter complexes and \cite{SW3} for several Dirac complex, respectively. 
Proposition \ref{boxj} implies that the associated Hodge-Laplacian operators in (\ref{hodge}) are all uniformly elliptic differential operators of fourth order with constant coefficients, i.e. there exists a constant $C > 0$ such that
\begin{align*}
C^{-1}|   {\xi}|^4I_{r_j}\leq L_j(   {\xi})\leq C |  {\xi}|^4I_{r_j}.
\end{align*}
\begin{prop}{\rm(\cite[Proposition 2.4.7]{Grafakos})}\label{PV} Let $K\in C^\infty\left(\mathbb R^{16}\setminus\{\mathbf 0\}\right)$   be a homogeneous function of degree $l-16,$ and  let $\mathbf K$ be the operator defined by $\mathbf K\phi = \phi *K.$
Then, for $\phi\in C_c^\infty\left(\mathbb R^{16}\right)$,  $A_1,\dots,A_l=1, 2,j_1,\dots, j_l=0,\dots, 7,$
\begin{align}\label{3.2}
\partial_{x_{A_1}^{j_1}}\dots\partial_{x_{A_l}^{j_l}}(\mathbf K\phi)=P.V.\left(\phi*\partial_{x_{A_1}^{j_1}} \dots\partial_{x_{A_l}^{j_l}} K\right) +a_{A_1j_1\dots A_lj_l}\phi,
\end{align}
for some constant  $a_{A_1j_1\dots A_lj_l}.$ Each term in {\rm (\ref{3.2})} is $C^\infty$ and the identity holds as $C^\infty$ functions.

Moreover, $\partial_{x_{A_1}^{j_1}}\dots\partial_{x_{A_l}^{j_l}}  K$  is a Calder\'on--Zygmund kernel on $\mathbb R^{16}.$ The singular integral operator $f\rightarrow P.V.$ $ \left(f *\partial_{x_{A_1}^{j_1}}\dots\partial_{x_{A_l}^{j_l}}K\right)$ is bounded
on $L^p$ for $1 < p <\infty.$ {\rm (\ref{3.2})} holds as $L^p$ functions.
\end{prop}
Let $\widetilde G_j(   {\xi})=L_j(   {\xi})^{-1}$ for $  {\xi}\in\mathbb O^{2}\setminus\{\mathbf 0\},$ and  let $G_j$ be the inverse Fourier transform  of $\widetilde G_j.$    It is a smooth  homogeneous function of degree $-12$ in $\mathbb R^{16}\setminus\{\mathbf 0\}$ by   \cite[Theorem 7.1.16 and 7.1.18]{Hormander}.  Let $\mathbf G_j$ be the convolution operator with kernel $G_j.$  We have the   following regularity  result for operator  $\mathbf G_j$.
\begin{prop}\label{GG}
The Laplacian $\Box_j$ has the inverse $\mathbf G_j$ in $L^2\left(\mathbb R^{16},\mathcal V_j\right)$ for each $j,$ which is a convolution operator with kernel $G_j.$ The kernel $G_j$ is a smooth  homogeneous function  of degree $-12.$  Moreover, $\mathbf G_j$ can be extended to a bounded linear operator from $L^p\left(\mathbb R^{16},\mathcal V_j\right)$ to $W^{4,p}\left(\mathbb R^{16},\mathcal V_j\right),$ for $1 < p <\infty.$ For  a $\mathcal V_j$-valued function $f$   in $C_c^l\left(\mathbb R^{16},\mathcal V_j\right)$ and any non-negative integer $l,$ we have $\mathbf G_jf\in C^{l+3}\left(\mathbb R^{16},\mathcal V_j\right).$
\end{prop}
\begin{proof}
For   $f\in C_c^\infty\left(\mathbb O^{2},\mathcal V_j\right),$ let \begin{align}\label{gG}
\mathbf G_jf=\int_{\mathbb O^{2}}G_j(\mathbf x-\mathbf y)f(\mathbf y)\hbox{d}V(\mathbf y).
\end{align}   Then for $f\in C_c^\infty\left(\mathbb O^{2},\mathcal V_j\right),$ \begin{align}\label{Gj}
\mathbf G_j\Box_jf=\Box_j\mathbf G_jf=f,
\end{align}holds in $L^2,$ since   their Fourier transforms are the same as $L^2$ functions. Since $L_j(\xi)$ are symmetric matrices by Proposition \ref{boxj}, we see that $G_j(  {\xi})$ are also symmetric. It follows from Plancherel theorem that   $\mathbf G_j$ is formally self-adjoint in the following sense:
\begin{align*}
\left(\mathbf G_j\phi,\psi\right)_{L^2\left(\mathbb O^{2},\mathcal V_j\right)}=\left(\phi,\mathbf G_j\psi\right)_{L^2\left(\mathbb O^{2},\mathcal V_j\right)},
\end{align*}
for any $\phi,\psi\in C_c^\infty\left(\mathbb O^{2},\mathcal V_j\right).$

By   Proposition \ref{PV}, the convolution with a homogeneous function of degree $-16+l$ can be extended to a bounded linear operator from $L^p\left(\mathbb O^{2}\right)$ to $W^{l,p}\left(\mathbb O^{2}\right).$  $\mathbf G_j$ is bounded from $L^2\left(\mathbb O^{2},\mathcal V_j\right)$ to $W^{4,2}\left(\mathbb O^{2},\mathcal V_j\right)$. So $\mathbf G_j$ can be extended to a bounded operator from $W^{-4,2}\left(\mathbb O^{2},\mathcal V_j\right)$ to $L^2\left(\mathbb O^{2},\mathcal V_j\right)$ by the duality argument. In particular, it is bounded from $W^{-2,2}\left(\mathbb O^{2},\mathcal V_j\right)$ to $L^2\left(\mathbb O^{2},\mathcal V_j\right).$  Thus (\ref{Gj}) holds as bounded linear operator from $L^2$ to $L^2,$ i.e.  $\mathbf G_j$ is the inverse operator of  Laplacian $\Box_j$ in $L^2.$

When $f\in C_c^{l}\left(\mathbb O^{2},\mathcal V_j\right),$    $\mathbf G_jf\in C^{l+3}\left(\mathbb O^{2},\mathcal V_j\right)$  by differentiation (\ref{gG}). The proposition is proved.
\end{proof}
 
 \vskip 3mm
{\it Proof of Theorem \ref{t31}.} We prove the case $j=0,$ the other cases are similar.
  Recall that  $\mathbf G_1$ is the inverse operator of $\Box_1$ on $L^2\left(\mathbb O^{2},\mathcal H_2(\mathbb O)\right)$ by Proposition \ref{GG}. We assume $u\in C_c^\infty\left(\mathbb O^{2},\mathcal H_2(\mathbb O)\right)$ first. Then
\begin{align*}
f:= \mathcal D_0^*\mathbf G_1u\in C^\infty\left(\mathbb O^{2},\mathbb R\right),
\end{align*}  by Proposition \ref{GG}. Note that
\begin{equation}\begin{aligned}\label{d1g1}
\mathcal D_1\Box_1u=&\mathcal D_1\left(\mathcal D_0\mathcal D_0^*+\left(\mathcal D_1^*\mathcal D_1\right)^2\right)u=\mathcal D_1\left(\mathcal D_1^*\mathcal D_1\right)^2u\\=&\left(\left(\mathcal D_1\mathcal D_1^*\right)^2+\mathcal D_2^*{\mathcal D_2}\right)\mathcal D_1u=\Box_2\mathcal D_1u,
\end{aligned}\end{equation}
   by $\mathcal D_1\mathcal D_0=0,{\mathcal D_2}\mathcal D_1=0.$
Now  for $u\in C_c^\infty \left(\mathbb O^{2},\mathcal H_2(\mathbb O)\right),$ we have
\begin{equation}\begin{aligned}\label{b2}
\Box_2 \left(\mathbf G_2\mathcal D_1u-\mathcal D_1\mathbf G_1u\right)=\mathcal D_1u-\mathcal D_1\Box_1\mathbf G_1u=0,
\end{aligned}\end{equation}
by using (\ref{d1g1}). For $u\in C_c^\infty \left(\mathbb O^{2},\mathcal H_2(\mathbb O)\right),$ it is direct to see that    $\left(\mathbf G_2\mathcal D_1u-\mathcal D_1\mathbf G_1u\right)\widehat{} \in L^2$  and $L_2( {\xi})\left(\mathbf G_2\mathcal D_1u-\mathcal D_1\mathbf G_1u\right)\widehat{} (  {\xi})=0$ for $  {\xi}\in\mathbb O^{2}$ by (\ref{b2}). Thus $\left(\mathbf G_2\mathcal D_1u-\mathcal D_1\mathbf G_1u\right)\widehat{} =0$ in $L^2.$ Consequently,    $ \mathbf G_2\mathcal D_1u=\mathcal D_1\mathbf G_1u $ in $L^2$ by the Plancherel theorem.  Moreover, $ \mathbf G_2\mathcal D_1u=\mathcal D_1\mathbf G_1u$ pointwise, since they are both smooth  by Proposition \ref{GG}.  Hence \begin{align}\label{GD}\mathbf G_2\mathcal D_1=\mathcal D_1\mathbf G_1,\end{align}as operators      from $C_c^\infty \left(\mathbb O^{2},\mathcal H_2(\mathbb O)\right)$ to $W^{2,2}\left(\mathbb O^{2}, \mathbb O^2\right)$ by Proposition \ref{GG}.
 Thus \begin{align}\label{DGf}\mathcal D_1\mathbf G_1u=\mathbf G_2\mathcal D_1u=0,\end{align}for $u\in C_c^\infty\left(\mathbb R^{16},\mathcal H_2(\mathbb O)\right),$ and so \begin{align}\label{DDDDG}
\mathcal D_0f=\mathcal D_0\mathcal D_0^*\mathbf G_1u=\left(\mathcal D_0\mathcal D_0^*+\left(\mathcal D_1^*\mathcal D_1\right)^2\right)\mathbf G_1u=u,
\end{align}
i.e. $f= \mathcal D_0^*\mathbf G_1u$ satisfies $\mathcal D_0f=u.$ 

On the other hand, the identity (\ref{GD}) holds as   bounded linear operators from the space  $L^2\left(\mathbb O^{2},\mathcal H_2(\mathbb O)\right)$ to $L^2\left(\mathbb O^{2}, \mathbb O^2\right)$, and so the identity (\ref{DGf}) holds for $u\in L^2\left(\mathbb O^{2},\mathcal H_2(\mathbb O)\right).$ Thus,  (\ref{DDDDG}) holds as $L^2$ functions for  $u\in L^2\left(\mathbb O^{2},\mathcal H_2(\mathbb O)\right)$ satisfying $\mathcal D_1u= 0$ in the sense of distributions. Thus    $f = \mathcal D_0^*\mathbf G_1u$ satisfies the  equation $\mathcal D_0f=u$.

 Suppose that $u$ is supported in a bounded domain $\Omega\Subset\mathbb R^{16}.$   By $\mathcal D_0f= u$, we  see that $f\in{\rm OPH}\left(\mathbb O^{2}\setminus \Omega\right).$   Note that   $G_j$ satisfies the following decay estimates
\begin{align}\label{partial}
\left|\partial_{x_{i_1}^{j_1}} \dots\partial_{x_{i_m}^{j_m}} G_j(\mathbf x)\right|\leq \frac{C_{i_1j_1\dots i_mj_m}}{|\mathbf x|^{12+m}},
\end{align}
for some constant $C_{i_1j_1\dots i_mj_m}>0$ depending on  $i_1,\dots,i_m=1,2,$ $j_1,\dots,$ $j_m=0,\dots,7.$  Thus  the integral kernel $K(\mathbf x)$ of $\mathcal D_0^*\mathbf G_1$ decays as $|\mathbf x|^{-14}$ for large $|\mathbf x|.$  Since  $f$ is compactly supported, we see that
\begin{align}\label{k}
\left|f(\mathbf x)\right|=\left|\int_{\mathbb O^{2}}K\left(\mathbf x-{\mathbf y}\right)u\left({\mathbf y}\right)\hbox{d}V\left({\mathbf y}\right)\right|\leq\frac{C}{\left(1+|\mathbf x|\right)^{14}},
\end{align}
for some constant $C>0.$
So  $\lim_{|\mathbf x|\rightarrow\infty}f(\mathbf x)=0.$ Since
$\Omega$ is bounded, its projection to $\mathbb R_{\mathbf x_1}^{ 8} $ is also bounded and so is contained in a ball centered at origin with   radius, say $M$. Therefore, 
for $\mathbf x=\left(\mathbf x_1,\mathbf x_2\right)$ with  $|\mathbf x_1|>M$, we have 
 $$\left(\left\{\mathbf x_1\right\}\times\mathbb R_{\mathbf x_2}^{8}\right)\cap\overline \Omega=\emptyset.$$
  Thus, $f\left(\mathbf x_1,\mathbf x_2\right)$ for $\left|\mathbf x_1\right|>M$ is  octonionic pluriharmonic.    In particular, it is harmonic  in variable $\mathbf x_2\in \mathbb{R}^{8}$ and   vanishes at infinity.  i.e. 
each component of $f\left(\mathbf x_1, \cdot\right)$ is a harmonic function on $\mathbb R_{\mathbf x_2}^8$.    Thus,  it is bounded,  and so $f\left(\mathbf x_1, \cdot\right)\equiv0$ for $|\mathbf x_1|>M$   by Liouville  theorem. Consequently, $f\equiv0$ on the unbounded connected component of $\mathbb O^{2}\setminus \Omega$ by the identity theorem for real analytic functions, since  $f$ is real analytic on $\mathbb O^{2}\setminus \Omega$ by Corollary \ref{p34}. The continuity of $f$ follows from the integral representation  formula as in (\ref{k}). This completes the proof of the theorem.
\qed

Now we have the Hartogs' extension phenomenon for OPH functions.
\begin{cor}\label{hart}
Suppose that   $\Omega$ is a bounded open set  in $\mathbb O^{2},$ and let $K$ be a compact subset of $\Omega$ such that $\Omega\setminus K$ is connected. Then for each  $f\in OPH(\Omega\setminus K),$   we can find  $F\in OPH(\Omega)$ such that $F=f$ in $\Omega\setminus K.$
\end{cor}

\subsection{Hartogs--Bochner extension theorem  for tangentially OPH  functions}
Now we can prove the  Hartogs--Bochner extension theorem  for tangentially OPH  function. 
\vskip 3mm
\noindent {\it Proof of Theorem \ref{HBE}.} 
By Proposition \ref{phat}, we can extend  $f$ to a smooth function $\tilde f$ on $\overline \Omega$ such that $\left.\tilde f\right|_{\partial \Omega}=f$ and $\mathcal D_0\tilde f$ is flat on $\partial \Omega.$ Then we can   extend $\mathcal D_0\tilde f$ by $0$ outside of $\overline \Omega$ to get a $\mathcal D_1$-closed element $F\in C_c^2\left(\mathbb O^{2},\mathcal H_2(\mathbb O)\right)$ supported in $\overline \Omega.$ Now  by Theorem \ref{t31}, there exists an  $H\in W^{2,2}\left(\mathbb O^{2},\mathbb R\right)$ vanishing  on the connected open set $\mathbb O^{2}\setminus\overline \Omega$   such that $F=\mathcal D_0H.$ Then $ \tilde f-H\in OPH(\Omega)$   and  gives us the required extension.
\qed
\appendix\section{ Proof of the  $\mathbb O_{\mathbb C}$-valued version of Lemma \ref{asso} }\label{app}
It is obvious  that  $  {\rm Re} (x) = {\rm Re} \left(\bar x\right)$ and  $  {\rm Re} (xy) = {\rm Re} (yx)$ for $x,y\in\mathbb O_{\mathbb C}$ by definition. Here we only prove the $\mathbb O_{\mathbb C}$-valued version of Lemma \ref{asso} (iii)-(iv).
\vskip 2mm  
\noindent{\it Proof of   Lemma 2.1 {\rm(iv)}.} Since $(xy)y=x(yy)$ and $(xy)x=x(yx)$ for $x,y\in\mathbb O_{\mathbb C}$ (\cite[Proposition 1.4.2]{SpringerV}), we have  
\begin{equation*}\begin{aligned}
(\bar xy)y=&((2{\rm Re}(x)-x)y)y=2{\rm Re}(x)yy-(xy)y=2{\rm Re}(x)yy-x(yy)=\bar x(yy),\\( x\bar y)y=&(x(2{\rm Re}(y)-y))y=2(x{\rm Re}(y))y-(xy)y=x(2{\rm Re}(y)y)-x(yy)=x(\bar yy),\\(\bar xy)x=&((2{\rm Re}(x)-x)y)x=2{\rm Re}(x)yx-(x y)x=2{\rm Re}(x)yx-x(yx)=\bar x(yx),\\( x\bar y)x=&(x(2{\rm Re}(y)-y))x=2(x{\rm Re}(y))x-(xy)x=x(2{\rm Re}(y)x)-x(yx)=x(\bar yx).
\end{aligned}\end{equation*} Similarly, we can prove the other cases. \qed

\begin{lem}\label{alem}For $x,y,z\in\mathbb O_{\mathbb C},$ we have \\
{\rm (i)} The associator $\{x,y,z\}$ changes sign if   conjugate any variable; \\{\rm (ii)} The associator $\{x,y,z\}$ is an alternating function of the three variables.
\end{lem}
\begin{proof}
(i) Since  $\{x,y,z\}=0$ if one of $x,y,z$ is complex,  
we can assume the variable to be conjugated is pure imaginary. Thus for example
$\{\bar x,y,z\} =\{-x,y,z\}=-\{x,y,z\}.$\\
(ii) By $\mathbb O_{\mathbb C}$-valued version of Lemma 2.1 (iv) we have already proved above, we have  $\{ x, x, y\} =0=\{x, y,   y\},$     polarisation gives us $\{ z+ x,z+ x, y\}=0=\{  x, y+z, y+z\},$ which lead
to $\{z,x,y\} + \{x,z,y\} = 0$, $\{x,y,z\} + \{x,z,y\} = 0.$ 
\end{proof}
\noindent {\it Proof  of Lemma 2.1 {\rm(iii)}.} Note that $\overline{xy}=\bar y\cdot\bar x$ for any $x,y\in \mathbb O_{\mathbb C}$ (\cite[Lemma 1.3.1]{SpringerV}),  we have  \begin{equation}\begin{aligned} \overline{\{x,y,z\}}=\overline{(xy)z-x(yz)}=\bar z\left(\bar y\bar x\right)-\left(\bar z\bar y\right)\bar x=-\left\{\bar z,\bar y,\bar x\right\}= {\{z,y,x\}}=-{\{x,y,z\}},\end{aligned}\end{equation}  by using  Lemma \ref{alem} (i)  (ii). Thus, ${\rm Re}((xy)z)={\rm Re}(x(yz)).$\qed

% ------------------------------------------------------------------------

\begin{thebibliography}{99}

 

\bibitem{Alesker1}
Alesker, S., Non-commutative linear algebra and plurisubharmonic functions of quaternionic variables, \emph{Bull. Sci. Math.} \textbf{127(1)}  (2003), 1-35. 
\bibitem{Alesker2}
Alesker, S., Quaternionic Monge--Amp\`ere equations, \emph{J. Geom. Anal. } \textbf{13(2)} (2003), 205-238.
\bibitem{Alesker3}
Alesker, S., Plurisubharmonic functions on the octonionic plane and Spin$(9)$-invariant valuations on convex sets, \emph{J. Geom. Anal.} \textbf{18(3)} (2008), 651-686.

\bibitem{AG}
Alesker, S. and Gordon, P.V., Octonionic Calabi--Yau theorem, {\it J. Geom. Anal.} \textbf{34 } (2024), Paper No. 293.



\bibitem{Andreotti}  
Andreotti, A., Hill, C., Lojasiewicz, S. and Mackichan, B.,  Complexes of differential operators, {\it Invent. Math.} \textbf{35 } (1976), 43-86.



\bibitem{AH1}
Andreotti, A. and Hill, C.,   Levi convexity and the Hans Lewy problem, \emph{Ann. Scuola Norm. Super. Pisa} \textbf{26} part I (1972),  325-363.

\bibitem{AH2}
Andreotti, A. and  Hill, C.,  Levi convexity and the Hans Lewy problem, \emph{Ann. Scuola Norm. Super. Pisa} \textbf{26} part II (1972), 747-806.


\bibitem{AN}
Andreotti, A. and Nacinovich, M.,  Noncharacteristic hypersurfaces for complexes of differential operators,
\emph{Ann. Mat. Pura Appl.} \textbf{125} (1980), 13-83.

\bibitem{AT}
Angella, D. and  Tomassini, A.,  On the $\partial\bar \partial$-lemma and Bott--Chern cohomology, \emph{Invent. Math.} 192 (2013), 71-81. 

\bibitem{Baez}
\newblock {Baez, J., The octonions,}
\newblock{ \emph{Bull. Amer. Math. Soc.}} \textbf{39(2)} (2002), 145-205.

\bibitem{Bedford}
Bedford, E.,
The Dirichlet problem for some overdetermined systems on the unit ball in $\mathbb C^n$,
\emph{Pacific J. Math.} \textbf{51} (1974), 19-25.
 \bibitem{BF}Bedford, E. and    Federbush, P., Pluriharmonic boundary values, {\it Tohoku Math. J.} \textbf{26} (1974), 505-511.
%\bibitem{B} Bedford, E., $\left(\partial\bar\partial\right)_b$ and the real parts of CR functions, {\it Indiana Univ. Math. J.} \textbf{29} (1980), 333-340.



 \bibitem{CSSS}Colombo, F., Sabadini, I., Sommen, F. and  Struppa, D.,  \emph{Analysis of Dirac systems and computational algebra,} Progress in Mathematical Physics,  \textbf{39}. Birkh\"auser  Boston, MA, 2004.




 \bibitem{Delanghe} Delanghe, R., Sommen, F. and  Sou\v{c}ek, {\it V., Clifford algebra and spinor-valued functions}, Mathematics and its Applications \textbf{53}, \newblock Kluwer Academic Publishers Group, Dordrecht, 1992.

 

\bibitem{CKS}
Colombo, F., Krau\ss har, R. and  Sabadini, I.,
Octonionic monogenic and slice monogenic Hardy and Bergman spaces, 
\emph{Forum Math.} \textbf{36(4)} (2024), 1031-1052.


\bibitem{CK}
Constales, D. and Krau\ss har, R., Octonionic Kerzman--Stein operators, \emph{Complex Anal. Oper. Theory.}
\textbf{15(6)} (2021),  Paper No. 104.



\bibitem{Fichera}
Fichera, G., Problemi al contorno per le funzioni pluriarmoniche,
Atti del Convegno celebrativo dell'80$^o$ anniversario della nascita di Renato Calapso, (Messina-Taormina, 1981), Ed. Veschi, Roma, 127-152.

 


\bibitem{Grafakos}
Grafakos, L., \emph{Classical and Modern Fourier Analysis,} Prentice Hall, Pearson, 2004.

\bibitem{Homander1}
 H\"ormander, L., \emph{An introduction to complex analysis in several variables}, Second revised edition,   North-Holland Mathematical Library, vol. \textbf{7},
North-Holland Publishing Co, Amsterdam-London, 1973, American Elsevier Publishing Co.,  New York. 

 \bibitem{Hormander}
{H\"ormander, L., \emph{The analysis of linear partial differential operators I,} in: Grundlehren der Mathematischen Wissenschaften, vol. \textbf{256} (2nd Edition), Springer-Verlag, Berlin (2003).}

 
\bibitem{HR} Huo, Q.-H. and  Ren, G.-B., Structure of octonionic Hilbert spaces with applications in the Parseval
equality and Cayley--Dickson algebras, \emph{J. Math. Phys.} \textbf{63(4)} (2022),   Article ID 042101.



 \bibitem{Krump}Krump, L. and Sou\v cek, V., The generalized Dolbeault complex in two Clifford variables, \emph{Adv. Appl. Clifford Algebr.} \textbf{17} (2007), 537-548.

 



\bibitem{Li}
Li, Y., Gilbert's conjecture and a new way to octonionic analytic functions from the Clifford analysis,
\emph{J. Geom. Anal.} \textbf{34(7)} (2024),   Paper No. 205.



\bibitem{KL}
Krau\ss har, R. and Legatiuk, D., Cauchy formulae and Hardy spaces in discrete octonionic analysis,
\emph{Complex Anal. Oper. Theory.} \textbf{18(1)} (2024), Paper No. 13.

\bibitem{MS}
Manogue, C. and Schray, J., Octonionic representations of Clifford algebras and triality, \emph{Found. Phys.} \textbf{26} (1996), 17-70.

\bibitem{MPT}
Maggesi, M., Pertici, D. and  Tomassini, G., Extension and tangential CRF conditions in quaternionic analysis,
\emph{Ann. Mat. Pura Appl.}   {\textbf{199}} (2020), 2263-2289.

\bibitem{Naci}
 Nacinovich, M., Complex analysis and complexes of differential operators,  Springer LNM \textbf{287} (1973), 105-195.
 
 \bibitem{Nacinovich85}  
Nacinovich, M.,   On boundary Hilbert differential complexes, {\it  Ann. Polon. Math.} \textbf{  46}  (1985), 213-235.

 




\bibitem{Poincare}
Poincar\'e, H.,  Sur les fonctions de deux variables, \emph{Acta Math.} \textbf{2}  (1883), 97-113.


 


 \bibitem{Sabadini} Sabadini, I., Sommen, F., Struppa, D. and  van Lancker, P., Complexes of Dirac operators in Clifford algebras, \emph{Math. Z.} \textbf{239} (2002), 293-320.

 

\bibitem{Schafer}
Schafer, R., \emph{Introduction to non-associative algebras}, Dover, New York, 1995. 

 


\bibitem{SW3}Shi, Y. and  Wang, W., Hartogs--Bochner extension for monogenic functions of several  vector variables and the Dirac complex,  \emph{Ann. Mat. Pura Appl.} \textbf{205(1)} (2026), 147-179.

\bibitem{SWW}Shi, Y., Wang, W. and Wu, Q.-Y.,  On monogenic functions and the Dirac complex of two vector variables,   {\emph{Adv. Appl. Clifford Algebr.}   \textbf{35(2)} (2025), Paper No. 18.} 
 

 
\bibitem{SpringerV}
Springer, T.A. and Veldkamp, F.D., {\it Octonions, Jordan Algebras and Exceptional Groups,}
Springer Monographs in Mathematics, Springer, Berlin, 2000.

\bibitem{WanW}
Wan, D.-R. and  Wang, W., On the quaternionic Monge--Amp\`ere operator, closed positive currents and Lelong-Jensen
type formula on the quaternionic space,
\emph{Bull. Sci. Math.} \textbf{141(4)} (2017), 267-311.

 
\bibitem{WangR}
Wang, H.-Y. and Ren, G.-B., Octonion analysis of several variables, \emph{Commun. Math. Stat.} \textbf{2} (2014), 163-185.

 


\bibitem{Wa10}{ Wang, W.},
The $k$-Cauchy--Fueter complexes, Penrose transformation  and Hartogs' phenomenon  for quaternionic $k$-regular functions,   {\it J. Geom. Phys.} {\bf 60} (2010), 513-530.



\bibitem{wang222}
Wang, W., The Neumann problem for the $k$-Cauchy--Fueter complex over $k$-pseudoconvex domains in $\mathbb R^4$ and the $L^2$ estimate, \emph{J. Geom. Anal.} \textbf{29} (2019),   1233-1258.

\bibitem{wang29}
Wang, W., On the boundary complex of the $k$-Cauchy--Fueter complex, \emph{Ann. Mat. Pura Appl.} \textbf{202} (2023), 2255-2291.
 
   \bibitem{XS}
Xu, Z.-H. and  Sabadini, I., Generalized partial-slice monogenic functions: the octonionic case,  to appear in \emph{Trans. Amer. Math. Soc.,} arXiv:2503.12409v4.
\end{thebibliography}
\end{document}